\documentclass{article}
\usepackage[margin=1in]{geometry}
\usepackage{natbib}
 \bibpunct[, ]{(}{)}{,}{a}{}{,}%
\usepackage{enumitem,authblk,amsmath,amssymb,amsthm,natbib}
\newcounter{uck}            
\newcommand{\newuck}{
	\refstepcounter{uck}
	\ensuremath{K_{\theuck}}
}
\newtheorem{theorem}{Theorem}

\newtheorem{lemma}[theorem]{Lemma}

\newtheorem{remark}{Remark}
\newcommand{\olduck}[1]{\ensuremath{K_{\ref{#1}}}}
\newcommand\numberthis{\addtocounter{equation}{1}\tag{\theequation}}
\newcommand{\R}{\mathbb{R}}
\newcommand{\oh}{\overline{h}}
\newcommand{\og}{\overline{g}}
\newcommand{\of}{\overline{f}}
\newcommand{\defn}{\overset{\text{def}}{=}}
\begin{document}
\title{Functional Central Limit Theorem for Two Timescale Stochastic Approximation}
\author[1]{Fathima Zarin Faizal}
\author[1]{Vivek S.\ Borkar}
\affil{Department of Electrical Engineering, IIT Bombay, Powai, Mumbai 400076, India}

\maketitle
\begin{abstract}
	Two time scale stochastic algorithms emulate singularly perturbed deterministic differential equations in a certain limiting sense, i.e., the interpolated iterates on each time scale approach certain differential equations in the large time limit when viewed on the `algorithmic time scale' defined by the corresponding step sizes viewed as time steps. Their fluctuations around these deterministic limits, after suitable scaling, can be shown to converge to a Gauss-Markov process in law for each time scale. This turns out to be a linear diffusion for the faster iterates and an ordinary differential equation for the slower iterates.
\end{abstract}
\section{Introduction} \label{sec: intro}
Stochastic approximation was introduced by \cite{robbins_munro} as a method to estimate $x^*$ satisfying the equation $h(x^*)=\theta    \ (:=$ the zero vector) under noisy measurements. Here $h(x)$ denotes the expected value at level $x$ of the response to a certain experiment and the user has access only to noisy measurements of $h(x)$ for a particular level $x$ in each experiment. Robbins and Monro provided the following scheme (which they dubbed `\textit{stochastic approximation}') for making successive experiments $x_1, x_2, \ldots$ so that under reasonable conditions, $x_n$ converges in probability to $x^*$:
\begin{align}
	x(n+1) = x(n) + a(n) \left(h(x_n)  + M_{n+1} \right). \label{eqn: single_sa_eqn}
\end{align}
The dynamical systems approach, a popular method for analyzing stochastic approximation, involves viewing the algorithm as a noisy Euler scheme of a limiting o.d.e.\ (for `\textit{ordinary differential equations}').  Here, the limiting o.d.e.\ is 
\begin{align}
	\dot{x}(t) = h(x(t)). \label{eqn: single_timescale_ode}
\end{align}
With some mild assumptions on the step sizes and the noise, it has been shown that the iterates capture the asymptotic behaviour of the limiting o.d.e.\ \cite{DerFrad}, \cite{Ljung}, \cite{Meerkov}. See  \cite{benaim1996}, \cite{Borkar2008} for a modern treatment. Over the years, this framework has found wide applications in communication networks, artificial intelligence and economic theory due to its incremental and adaptive nature. 

In this work, we consider two timescale stochastic approximation, where different parts of the iteration run at different stepsizes. This introduces the notion of different timescales in the algorithm where each timescale corresponds to a specific choice of stepsizes. Some recent applications of this algorithm include distributed optimization \cite{distributed_optimization1} and reinforcement learning \cite{borkar_meyn_SA_RL,konda_phd}. 

The classical functional central limit theorem (FCLT), also called Donsker's theorem, is an extension of the classical Lindeberg-L\'{e}vy central limit theorem to function spaces. Donsker's theorem shows the weak convergence of the laws of suitably interpolated and scaled simple symmetric random walk to the law of a Brownian motion. More generally, functional central theorems provide convergence in law of scaled interpolations of deviations of discrete processes from their deterministic or `fluid' limits, to a Gauss-Markov process. For the problem of interest here, viz., the two time scale stochastic approximation, while convergence guarantees and concentration bounds have been proved, a functional central limit theorem has not been shown. This is the gap we fill in here. 

A functional central limit theorem for the single timescale stochastic approximation can be found in \cite{Borkar2008}, Chapter 7. It shows that the distribution of the interpolated deviations of the iterates from the o.d.e.\ it tracks, when scaled by $\sqrt{a(n)}$ where $a(n)$ is the step size at time step $n$, converges to a zero mean Gauss-Markov process. A straightforward extension of this result to the two timescale case would be an analogous result for each timescale as in the single timescale case, viz., that the distribution of the interpolated deviations of the iterates from the o.d.e.\ of each timescale, when scaled by the square-root of the step size for that particular timescale, converges to a zero mean Gaussian process. In our results, we show that the correct formulation is: for each timescale, when the deviations of the iterates from the limiting o.d.e.\ are  scaled by the square-root of the step size of the \textit{faster} timescale, we achieve convergence in distribution to a zero mean Gaussian diffusion. 

The main novelty in our work is in characterizing an iterative equation for the fluctuations of the faster timescale. The other main challenge is in dealing with the additional fluctuation arising from the dependence of the iterates in the slower timescale on the iterates in the faster timescale. To control this deviation, a scaling by the square-root of the faster step size for both timescales was found to be required in our analysis and resulted in a functional central limit theorem that substantially differs from the aforementioned straightforward extension one might expect for the slower timescale (Theorem \ref{thm: main_thm} below). Finally, we also derive a central limit theorem from the FCLT proved in Theorem \ref{thm: main_thm_alpha}.

\subsection{Prior work} 
After the Robbins-Monro algorithm was introduced in \cite{robbins_munro}, similar and more powerful results have been proven for this class of iterative algorithms such as almost sure convergence, conditions for convergence, convergence rates and so on \cite{benaim1996}, \cite{Borkar2008}.

Several other limit theorems have also been proved for single timescale stochastic approximation. Rate of convergence for moments has been shown in \cite{gerencser}. An almost sure central limit theorem and a law of iterated logarithms were proved in \cite{pelletier1999} and \cite{pelletier1998} respectively. A functional central limit theorem for stochastic approximation with martingale noise can be found in Chapter 7 of \cite{Borkar2008}. Another functional central limit theorem for stochastic approximation with Markovian noise was proved in \cite{borkar_meyn_markov}.  (These references are only representative, see \cite{Duflo} for an excellent overview of the classical work on stochastic approximation.) For FCLT for single time scale, see \cite{Lai_Robbins}, \cite{solo}, \cite{benveniste_priouret_metivier},  \cite{esfahani_heunis}, \cite{borkar_meyn_markov}.

Several analogous results exist in the literature for two timescale stochastic approximation. Almost sure convergence was proved in \cite{borkar_conv_two_ts}. Concentration bounds for two timescale stochastic approximation were proved in \cite{pattathil_2018}. Sufficient conditions for the stability of two timescale stochastic approximation were proved in \cite{lakshminarayanan_stability}. A central limit theorem for two timescale stochastic approximation when the iterations are linear was proved in \cite{konda_linear}, \cite{Basak}. A central limit theorem for suitably scaled iterates of two timescale stochastic approximation was proved in \cite{mokkadem_pelletier}. See \cite{jordan_fclt} for a recent work on FCLT for Q-learning with Polyak-Ruppert averaging.

\subsection{Problem description}
We consider the two timescale stochastic approximation scheme involving martingale and Markov noise. Precisely, we consider the following iterative equations for $\{x_k\}$, $\{y_k\}$ where $x_k \in \mathbb{R}^{d_1},  y_k \in \mathbb{R}^{d_2}$, $h: \mathbb{R}^{d_1}\times\mathbb{R}^{d_2} \times S \rightarrow \mathbb{R}^{d_1},g : \mathbb{R}^{d_1}\times\mathbb{R}^{d_2} \times S\rightarrow \mathbb{R}^{d_2}$ and $k \geq 0$: 
\begin{align}
	x_{k+1} & =x_k + a(k) \left ( h(x_k, y_k,Y_k) + M_{k+1}^{f} \right ), \label{eqn: iteration1}\\
	y_{k+1} & =y_k + b(k) \left ( g(x_k, y_k,Y_k) + M_{k+1}^{s} \right ), \label{eqn: iteration2}
\end{align}
such that
\begin{align}
	\sum_na(n)=\sum_nb(n)=\infty, \ \sum_n(a(n)^2+b(n)^2)<\infty, \ 	\frac{b(n)}{a(n)} \rightarrow 0. \label{eqn: bk_by_ak}
\end{align}
Here, $\{Y_k\}$ is the so called Markov noise taking values in a finite space $S$ and controlled by $\{x_n, y_n \}$ (see Assumption \ref{assumption: markov_conditional_prob} for the definition). Let $\pi_{x, y}$ be the unique stationary distribution of an irreducible Markov chain on $S$ with  transition probability $p_{x,y}( \cdot | \cdot )$ that governs the evolution of $\{Y_n\}$ as described in Assumption \ref{assumption: markov_conditional_prob}. 

The last condition in \eqref{eqn: bk_by_ak} essentially puts $\{x_k\}$ on a much faster timescale compared to $\{y_k\}$. This also implies that there are essentially three timescales at play here. The process $\{Y_n\}$ runs on the natural timescale, i.e., the clock indexed by $n = 0, 1, 2, \cdots$. On the other hand, both iterates above run on slower timescales defined by time increments $a(n), b(n), n \geq 0,$ resp. 

The equations \eqref{eqn: iteration1} and \eqref{eqn: iteration2} can be thought of as a noisy Euler scheme for the coupled (`singularly perturbed') o.d.e.s 
\begin{align*}
	\dot{x}(t) &= \frac{1}{\epsilon} \sum_i\pi_{x(t),y(t)}(i)h(x(t),y(t),i) = \frac{1}{\epsilon} \overline{h}(x(t),y(t)), \\
	\dot{y}(t) &= \sum_i\pi_{x(t), y(t)}(i)g(x(t),y(t),i) = \overline{g}(x(t),y(t)),
\end{align*}
where $0 < \epsilon \downarrow 0$. Note that the averaging over $\{Y_n\}$ which moves on a faster time scale can be seen here in the averaging of the right hand side w.r.t.\ its stationary distribution parametrized by $(x,y)=(x(t),y(t))$ at time $t$. Assume that for a fixed $y$, the o.d.e.
\begin{align}
	\dot{x}(t) &= \frac{1}{\epsilon} \sum_i\pi_{x(t),y}(i)h(x(t),y,i), \label{eqn: first_timescale_ode}
\end{align}
has a globally asymptotically stable equilibrium $\lambda(y)$ for a Lipschitz $\lambda(\cdot)$,  and the o.d.e.
\begin{align}
	\dot{y}(t) &= \sum_i\pi_{\lambda(y(t)), y(t)}(i)g(\lambda(y(t)),y(t),i) \label{eqn: second_timescale_ode}
\end{align}
has a globally asymptotically stable equilibrium $y^{*}$. Since $\epsilon$ does not affect the trajectories of $x(\cdot)$ and affects only the speed with which they are traversed, $\lambda(\cdot)$ is independent of $\epsilon$.  If $\epsilon$ is sufficiently small, then $y(\cdot)$ can be considered to be quasi-static relative to $x(\cdot)$. Thus $x(t)$ essentially tracks $\lambda(y(t))$ from \eqref{eqn: first_timescale_ode} and $y(t)$ approximately tracks the solution of the o.d.e.\ \eqref{eqn: second_timescale_ode}. We thus expect $(x(t),y(t))$ to approximately converge to $(\lambda(y^{*}), y^{*})$ and by extension, $(x_n, y_n)$ to converge almost surely to $(\lambda(y^*), y^*)$. This is indeed so, refer to \cite{Borkar2008}, Chapter 8, for details. Note, however, that throughout the rest of the paper, we do not assume that \eqref{eqn: second_timescale_ode} has a globally asymptotically stable equilibrium point. 

Our aim is to prove a functional central limit theorem for both iterates along the lines of the functional central limit theorem for single timescale stochastic approximation, which will be discussed in more detail in section \ref{sec: main_results}. 

\subsection{Technical assumptions}
We make the following assumptions to prove Theorem \ref{thm: main_thm}:
\begin{enumerate}[label=(A\arabic*)]
	\item $\{a(k)\}, \{b(k)\} \subset (0,1)$, and
	\begin{align*}
		\sum_{k}a(k) = \sum_{k} b(k) &= \infty, \\
		\sum_{k} \left ( a(k)^2 + b(k)^2 \right ) &< \infty, \\
		\frac{b(k)}{a(k)^{3/2}} &\overset{n \uparrow \infty}{\rightarrow} 0. 
	\end{align*}
	Also, $0 < b(k) \leq a(k)^{3/2} < a(k) < 1 \, \forall \, k$.
	
	\item \label{assumption: alpha_beta} $\varphi := \lim_{n\uparrow\infty}(\frac{1}{a(n+1)} - \frac{1}{a(n)})$ exists and equals 0. The case when $\varphi \neq 0$ has been handled separately in Section \ref{sec: alpha}. An example of such a sequence would be $a(n) = \frac{C}{(n+1)^c}, c \in \left(\frac{1}{2},1\right)$. As mentioned in \cite{Borkar2008}, Chapter 7, some consequences of this assumption are:
	\begin{align*}
		\underset{n \uparrow \infty}{\lim} \ \frac{a(n)}{a(n+1)} &= 1, \\
		\underset{n \uparrow \infty}{\lim} \	\frac{\left( a(n) - a(n+1)\right)^2}{a(n) a(n+1)^2} &= 0.
	\end{align*}
	We also assume that $\vartheta := \lim_{n\uparrow\infty}(\frac{1}{b(n+1)} - \frac{1}{b(n)})$ exists.
	\item \label{assumption: bn_omega} \begin{align*}
		\underset{n \uparrow \infty}{\lim} \ \frac{a(n)}{b(n)} \left( \frac{1}{a(n+1)} - \frac{1}{a(n)}\right) = 0.
	\end{align*}
	\item \label{assumption: an_decreasing} The sequence $a(n)$ decreases monotonically to 0.
	\item \label{assumption: LemmaA.28} For any $n \in \mathbb{N}$ and $T>0$, let $m^{f-}(n) = \max \left\{m \leq n : t^f(m) \leq t^f(n) - T \right\}$, where $t^f(n) = \sum_{i=1}^{n-1} a(i)$. Then for any $n$,
	\begin{align*}
		\underset{T \uparrow \infty}{\lim} \frac{a(m^{f-}(n,T))}{a(n) e^T} < \zeta(T) \rightarrow 0,
	\end{align*}
	where $\zeta(T)$ is independent of $n$ and is a decreasing function of $T$. In the above, we added a second argument to $m^{f-}(\cdot)$ to highlight its dependence on $T$.
	\item\label{assumption: bdd_fourth_moment} $\sup_n\|x_n\| < \infty$  a.s., \
	$\sup_nE\left[\|x_n\|^4\right] < \infty$. Also, $\sup_n\|y_n\| < \infty$  a.s., \
	$\sup_nE\left[\|y_n\|^4\right] < \infty$. Such stability assumptions are common in the stochastic approximation literature to prove convergence and limit theorems. One can consult Chapter 4 of \cite{Borkar2008} for examples of stability criteria. 
	
	\item $\{M_{k+1}^{i}\}, i \in \{f,s\}$ are martingale difference sequences with respect to the increasing $\sigma$-fields
	\begin{align*}
		\mathcal{F}_k := \sigma(x_m, y_m, Y_m, M_m^{i}; i \in \{f,s\}; m \leq n), n \geq 0.
	\end{align*}
	That is, $\mathbb{E} \left[ M^{i}_{n+1} | \mathcal{F}_n \right] = 0$ a.s. for $ i \in \{f,s\}$ and $n \geq 0$.
	
	\item\label{assumption: martingale} $\{M_{k+1}^{i}\},  i \in \{f,s\}$ also satisfy:
	\begin{eqnarray*}
		E\left[M_{n+1}^{f}{M_{n+1}^{f}}^{\mathrm{T}} | M_k^{f}, M_k^{s}, x_k, y_k, Y_k, k \leq n\right] &=& Q_f(x_n, y_n), \\
		E\left[M_{n+1}^{s}{M_{n+1}^{s}}^{\mathrm{T}} | M_k^{s},M_k^{f}, x_k, y_k, Y_k, k \leq n\right] &=& Q_s(x_n, y_n), \\
		E\left[\|M_{n+1}^{f}\|^4 | M_k^{f}, M_k^s, x_k, y_k, Y_k, k \leq n\right] &\leq& K'\left(1 + \|x_n\|^4 + \|y_n\|^4 \right), \\
		E\left[\|M_{n+1}^{s}\|^4 | M_k^{f}, M_K^s, x_k, y_k, Y_k, k \leq n\right] &\leq& K'\left(1 + \|x_n\|^4 + \|y_n\|^4 \right),
	\end{eqnarray*}
	where $K' > 0$ is a suitable constant and $Q_f': \R^{d_1}  \times \R^{d_2} \rightarrow \R^{d_1\times
		d_1}$, $Q_s: \R^{d_1} \times \R^{d_2} \rightarrow \R^{d_2\times
		d_2}$ are positive definite matrix-valued Lipschitz functions such
	that the least eigenvalues of $Q_i(\cdot, \cdot), i \in \{f,s\}$ are bounded away from zero
	uniformly in its arguments. In particular, the last assumption combined with \ref{assumption: bdd_fourth_moment} gives you that the fourth moments of the martingale noise are bounded. 
	\item \label{assumption: hg} The functions $h : \R^{d_1} \times \R^{d_2} \times S \rightarrow \R^{d_1}$ and $g :R^{d_1} \times \R^{d_2} \times S \rightarrow \R^{d_2}$ are uniformly Lipschitz in each of their first two arguments for each value of $Y_n$, i.e.,  for some $L >0$ and $\forall \ i \in S,x,x' \in \R^{d_1}, y,y'\in\R^{d_2},$ 
	\begin{align*}
		\| h(x,y,i) - h(x', y', i) \| &\leq L(\|x-x'\| + \|y-y'\|),  \\
		\| g(x,y,i) - g(x', y', i) \| &\leq L(\|x-x'\| + \|y-y'\|).
	\end{align*}
	We also assume that $\oh : \R^{d_1} \times \R^{d_1} \rightarrow \R^{d_1}$ and $\og : \R^{d_2} \times \R^{d_2} \rightarrow \R^{d_2}$ are continuously differentiable and their Jacobian matrices $\nabla_x\oh,\nabla_y\oh, \nabla_x\og,\nabla_y\og$ respectively are uniformly bounded and uniformly Lipschitz continuous. Moreover, the functions $h(\cdot, \cdot,\cdot)$ and $g(\cdot, \cdot,\cdot)$ are uniformly bounded. 
	\item \label{assumption: fast_ode_convergence} The o.d.e. \eqref{eqn: first_timescale_ode} has a globally asymptotically stable equilibrium $\lambda(y)$ where $\lambda(\cdot)$ is continuously differentiable with bounded derivatives.
	\item \label{assumption: fast_ode_convergence_martingale} For any $y \in \R^{d_1}$, the solution of the o.d.e.\ \eqref{eqn: first_timescale_ode} satisfies $\left \| x(t) - \lambda(y) \right \| \leq c e^{-dt} \| x(0) - \lambda(y) \|$, for some $c,d>0$ (i.e., the equilibrium $\lambda(y)$ is exponentially stable). 
	\item\label{assumption: markov_conditional_prob} $\{Y_n\}_{n \geq 0}$ is such that for $i \in S$ and $n \geq 0$, 
	\begin{align*}
		\mathbb{P}(Y_{n+1} =i | Y_m, x_m, y_m, m \leq n) &= \mathbb{P}(Y_{n+1} =i | Y_n, x_n, y_n), \\
		&= p_{x_n, y_n} (i|Y_n),
	\end{align*}
	where for each fixed $x,y$, $p_{x,y}(\cdot | \cdot)$ is the transition probability of an irreducible Markov chain on the finite state space $S$ with the unique stationary distribution $\pi_{x,y}$. Furthermore, the maps $(x,y) \mapsto p_{x,y}(\cdot|\cdot)$ are continuously differentiable with bounded derivatives. 
	\item\label{assumption: martingale_markov_cond_independence} The Markov noise $Y_n$ and the martingale noises $\{M^i_{n+1}, i = f,s\},$ are conditionally independent of each other and the past, i.e., $\{x_0,y_0, Y_k,M^i_{k+1}, i = 1,2; k \geq 0\}$, given the current iterates $(x_n,y_n)$.
	Furthermore, the conditional covariance matrices of the martingale noises given the current iterates are positive definite and Lipschitz.
	
\end{enumerate}
\begin{remark}
	An example of such $\{a(n)\}$ and $\{b(n)\}$ that satisfy (A1.1)-(A1.4) is 
	\begin{align*}
		a(n) &= \frac{C_1}{(n+1)^\rho}, \\
		b(n) &= \frac{C_2}{(n+1)^{3\rho/2 + \epsilon}},
	\end{align*}
	where $\rho \in \left(\frac{1}{2},\frac{2}{3}\right)$ and $\epsilon\in (0,1 - 3\rho/2)$. Thus,
	\begin{align*}
		\frac{a(n)}{b(n)} \left( \frac{1}{a(n+1)} - \frac{1}{a(n)}\right) &= \frac{(n+1)^{\epsilon+3\rho/2 } }{C_2} \left(  \left(1+\frac{1}{n+1}\right)^{\rho}-1 \right) \overset{n \uparrow \infty}{\rightarrow} 0.
	\end{align*}
\end{remark}
\begin{remark}
	For the example given in Remark 1, it follows that for any $n$,
	\begin{align*}
		\frac{a(m^{f-}(n))}{a(n)} \leq b (1+T)^{\rho/(1-\rho)},
	\end{align*}
	where $b$ is a constant independent of $n$ and $T$ (Lemma A.28 of \cite{borkar_meyn_SA_RL}). It follows that \ref{assumption: LemmaA.28} is also satisfied. 
\end{remark}
\subsection{Notation} \label{sec: notation}
In this section, we define the various terms required to state our main result. We first define the terms required to state the functional central limit theorem for single timescale stochastic approximation shown in Chapter 7 of \cite{Borkar2008}. We begin with some background. An o.d.e.\ is said to be well-posed if there exists a unique solution for every initial condition that is continuous with respect to the latter. Recall that an invariant set $A$ for a well-posed o.d.e.\ is said to be internally chain transitive if given any $x,y\in A$ and $\epsilon, T > 0$, we can find an $n\geq 2$ and points $x_0=x,  x_1, \cdots x_{n-1}, x_n = y,$ such that there exist  trejectories $\breve{x}_i(\cdot)$ of the  o.d.e.\ for $0\leq i < n,$ satisfying: $\breve{x}_i(0)$ is in the $\epsilon$-neghborhood of $x_i$ and  $\breve{x}_i(t)$ is in the $\epsilon$-neighborhood of $x_{i+1}$ for some $t \geq T$.

The broad idea behind the o.d.e.\ approach is to construct an interpolated trajectory on the algorithmic timescale from the iterates and show that it asymptotically approaches the solution set of the limiting o.d.e. The following theorem, which builds upon a classic result of  Benaim \cite{benaim1996} in the single time scale case, characterizes the asymptotic behaviour of $\{x_n\}, \{y_n\}$:
\begin{theorem}[\cite{Borkar2008}, Theorem 8.1]  \label{thm: as_limit}
	Almost surely, the sequences $\{x_n\}, \{y_n\}$ generated by \eqref{eqn: iteration1}, \eqref{eqn: iteration2} respectively converge to $\{(\lambda(y), y) : y \in \chi \}$, where $\chi$ is a (possibly sample path dependent) connected internally chain transitive invariant set of \eqref{eqn: second_timescale_ode}. 
\end{theorem} 

Consider the iterative equation in $\R^d$ given by \eqref{eqn: single_sa_eqn}. Define the algorithmic timescale  
\begin{align*}
	t(n) = \sum_{m=0}^{n-1} a(m), \ n \geq 1; \ t(0)=0. 
\end{align*}
Note that this timescale runs slower than the normal timescale as $1>a(n)> 0, a(n)\to 0$. Fix a value of $T>0$. Let $m(n):=\min \{m \geq n : t(m) \geq t(n) + T\}, n \geq 1$. We also define $x^n(t), t \geq t(n)$, to be the unique solution of \eqref{eqn: single_timescale_ode} starting at $t(n)$: 
\begin{align*}
	\dot{x}^n(t) = h(x^n(t)), t \geq t(n); \ x^n(t(n)) = x_n. 
\end{align*}
In \cite{Borkar2008}, the functional central limit theorem for the single timescale case is stated for linear interpolations of the scaled fluctuations of the iterates from the limiting o.d.e. For a particular $n \in \mathbb{N}$, the scaled fluctuations are defined as: for  $j \geq n$,
\begin{align*}
	z_j^n := \frac{x_j - x^n(t(j))}{\sqrt{a(j)}}. 
\end{align*}
These are then linearly interpolated over the algorithmic timescale to form the piecewise linear function $z^n(t), t \in [t(n), t(n)+T]$. Specifically, $z^n(t(j))=z_j^n$ and $z^n(t)$ is a linear interpolation of $z_j^n$ and $z_{j+1}^n$ for $t \in [t(j), t(j+1)]$ and $n \leq j < m(n)$. For ease of notation, we define $\tilde{z}^n(t) = z^n(t(n)+t)$. 

Analogously, to study the functional central limit theorem for the two timescale case, we define the two `algorithmic timescales'
\begin{align*}
	t^f(n) &= \sum_{m=0}^{n-1} a(m), \ n \geq 1; \ t^f(0)=0,\\
	t^s(n) &= \sum_{m=0}^{n-1} b(m), \ n \geq 1; \ t^s(0)=0,
\end{align*}
with
\begin{align*}
	m^f(n) &:=\min \{m \geq n : t^f(m) \geq t^f(n) + T\}, n \geq 1, \\
	m^s(n) &:=\min \{m \geq n : t^s(m) \geq t^s(n) + T\}, n \geq 1.
\end{align*}
Clearly, $t^i(m^i(n)) \in \left[t^i(n) + T, t^i(n) + T + 1\right], \ i \in \{f,s\}$. Define $\overline{f}(y)=\overline{g}(\lambda(y),y)$. Let $y^n(t)$ be the unique solution of the differential equation
\begin{align*}
	\dot{y}^n(t) = \overline{f}(y^n(t)), t \geq t^s(n); \ y^n(t^s(n)) = y_n. 
\end{align*}
The fluctuations that we consider for the functional central limit theorem are the following. For $n \geq 1$, let 
\begin{align*}
	\beta_n &=\lambda(y_n) - x_n= v_n -x_n,\\
	u_j &= \frac{\beta_j}{\sqrt{a(j)}}, \\
	\alpha_j^n &= y_j - y^n(t^s(j)), \\
	w_j^n &= \frac{\alpha_j^n}{\sqrt{a(j)}}.
\end{align*}
We then construct linear interpolations of these fluctuations over an interval of length $\approx T$ as before. The piecewise linear function $u^n(t), t \in [t^f(n), t^f(n)+T]$ is such that $u^n(t^f(j)) = u_j$ and is linearly interpolated on the intervals $[t^f(j), t^f(j+1)], n \leq j \leq m^f(n)$. Similarly, the piecewise linear function $w^n(t), t \in [t^s(n), t^s(n)+T]$ is such that $w^n(t^s(j)) = w_j$ and is linearly interpolated on the intervals $[t^s(j), t^s(j+1)], n \leq j \leq m^s(n)$. For ease of notation, we also define $\tilde{u}^n(t)=u^n(t^f(n)+t), t \in [0,T]$, $\tilde{y}^n(t)=y^n(t^s(n)+t)$ and $\tilde{w}^n(t) = w^n(t^s(n)+t), t \in [0,T]$. Also denote by $\overline{y}^n(t), t \in [0,T]$ the linearly interpolated version of $\{y_n\}$ on the slow algorithmic timescale, i.e., for $n \leq j \leq m^s(n)$, $\overline{y}^n(t^s(j)-t^s(n)) = y_j$ and linearly interpolated otherwise. 

\subsection{Main results} \label{sec: main_results}
We first reproduce the functional central limit theorem for single timescale stochastic approximation from \cite{Borkar2008} and \cite{borkar_meyn_markov} for comparison. 
\begin{theorem}[Theorem 7.1, \cite{Borkar2008}] \label{thm: fclt_single}
	Consider the single timescale stochastic approximation iteration as in \eqref{eqn: single_sa_eqn}. The limit points in law $(z^*(\cdot), x^*(\cdot))$ of $\{\tilde{z}^n(\cdot), \tilde{x}^n(\cdot)\}$ are such that $x^*(\cdot)$ is a solution of the o.d.e.\ $\dot{x}(t)=h(x(t))$ belonging to an internally chain transitive invariant set thereof, and $z^*(\cdot)$ satisfies 
	\begin{align*}
		z^*(t) = \int_{0}^{t} \left( \nabla h(x^*(s)) + \frac{\varphi}{2} I \right) z^*(s) ds + \int_{0}^{t} G(x^*(s)) dB(s),
	\end{align*}
	where $G(x) \in \R^{d \times d}$ for $x \in \R^d$ is a positive semidefinite, Lipschitz in $x$, square-root of the matrix $Q(x)$\footnote{Such a $G$ can always be found, see \cite{SV}, Section 5.3.}. Here, $Q(x)$ satisfies
	\begin{align*}
		E \left[M_{n+1} M_{n+1}^{\intercal} | M_i, x_i, i \leq n \right] = Q(x_n). 
	\end{align*} 
\end{theorem}
Thus $z^*(\cdot)$ is the solution of a linear stochastic differential equation and is a zero mean Gaussian process. Our main result is summarized in the following theorem (notation defined in Section \ref{sec: notation}).
\begin{theorem} \label{thm: main_thm}
	The limit points in law $(u^*(\cdot),w^*(\cdot), y^*(\cdot))$ of $\{\tilde{u}^n(\cdot), \tilde{w}^n(\cdot), \tilde{y}^n(\cdot) \}$ are such that $u^*(\cdot)$ satisfies
	\begin{align}
		u^*(t)  = \int_{0}^{t} \left ( \nabla_x \oh(\lambda(y^*(s)),y^*(s)) \right )u^*(s) ds + \int_{0}^{t} G(\lambda(y^*(s)),y^*(s)) dB(s), \label{eqn: fast_sde}
	\end{align}
	where $G(\cdot,\cdot)$ is a positive semidefinite, Lipschitz matrix, and $w^*(\cdot)$ satisfies
	\begin{align}
		w^*(t)  = \int_{0}^{t} \left( \nabla \of(y^*(s))w^*(s) +\nabla_x \og(\lambda(y^*(s)),y^*(s)) u^*(s)  \right) ds. \label{eqn: slow_ode}
	\end{align}
	and $y^*(\cdot)$ is a trajectory belonging to an internally chain transitive invariant set of the o.d.e.\ $\dot{y}(t)=\of(y(t))$.
\end{theorem}
This result is stated under our assumption that $\varphi = 0$. Later in this work we consider $\varphi \neq 0$. With this in mind, some key calculations in what follows have been carried out without assuming $\varphi=0$, to underscore the fact that they remain valid otherwise. The impact of the statistics of the Markov chain $\{Y_n\}$ can be seen in the averaging with respect to its stationary distribution in the functions $\oh(\cdot, \cdot)$, $\of(\cdot)$ and $\og(\cdot,\cdot)$.

In particular, $u^*(\cdot)$ is a solution of a linear stochastic differential equation and $w^*(\cdot)$ is a solution of an \textit{ordinary} differential equation. One can obtain the following expression for $u^*(\cdot)$ using the variation of constants formula:
\begin{align*}
	u^*(t) = \int_{0}^{t} \Phi_1 (t,s) G(x^*(s)) dB(s).
\end{align*}
Here, $\Phi_1(t,s), t \geq s \geq 0$ is a solution of the matrix differential equation
\begin{align*}
	\frac{d}{dt} \Phi_1(t,s) = \left( \nabla_x \oh(\lambda(y^*(s)), y^*(s)) + \frac{\varphi}{2}I \right) \Phi_1(t,s), t \geq s; \ \Phi_1(s,s) = I.
\end{align*} 
As $y^*(\cdot)$ is a deterministic trajectory, $\Phi_1(\cdot, \cdot)$ is also deterministic. Thus $u^*(\cdot)$ is a zero mean Gaussian process. Similarly, one can obtain the following expression for $w^*(\cdot)$ using the variation of constants formula:
\begin{align*}
	w^*(t) = \int_{0}^{t} \Phi_2 (t,s) \nabla \of( y^*(s)) u^*(s) ds,
\end{align*}
where $\Phi_2(t,s), t \geq s \geq 0$ is a solution of the matrix differential equation
\begin{align*}
	\frac{d}{dt} \Phi_2(t,s) =  \nabla \of (y^*(s))  \Phi_2(t,s), t \geq s; \ \Phi_2(s,s) = I.
\end{align*} 
Note that $\Phi_2(\cdot, \cdot)$ is deterministic. As $u^*(\cdot)$ is a zero mean Gaussian process, so is $w^*(\cdot)$.	
\section{Proofs}
In this section, we will provide a proof of Theorem \ref{thm: main_thm}. 
\subsection{Preliminaries}
In this section, we discuss some preliminary results that will be required for the proof.
\subsubsection{Decomposition}
The iterative equations \eqref{eqn: iteration1} and \eqref{eqn: iteration2} can be rewritten in the following manner:
\begin{align}
	x_{n+1} &= x_n + a(n) \left[\overline{h}(x_n, y_n) + \Delta^f_{n}(y_n) + M^f_{n+1} \right], \label{eqn: iteration1_decomp}\\
	y_{n+1} &= y_n + b(n) \left[\overline{f}(y_n) + \Delta^s_n(x_n) + p(x_n, y_n) + M^s_{n+1} \right],\label{eqn: iteration2_decomp}
\end{align}
where
\begin{align}
	\Delta_n^f(y) &= h(x_n, y, Y_n) - \overline{h}(x_n, y), \label{eqn: m_decomp_f}\\
	\Delta_n^s(x) &= {g}(x, y_n, Y_n) - \overline{g}(x, y_n), \label{eqn: m_decomp_s} \\
	p(x_n, y_n) &= \og(x_n, y_n)-\of(y_n).
\end{align}
Let $V^f : \R^{d_1} \times \R^{d_2} \times S \rightarrow \R^{d_1}$, $V^s : \R^{d_1} \times \R^{d_2} \times S \rightarrow \R^{d_2}$ be the solutions of the parametrized Poisson equations in $V^{(\cdot)}(x,y,\cdot)$, paramtetrized by $x,y$, given by:
\begin{align*}
	V^f(x,y,i) &= h(x,y,i) - \sum_{j \in S} \pi_{x,y}(j) h(x,y,j) + \sum_{j \in S} p_{x,y}(j|i)V^f(x,y,j), \\
	V^sf(x,y,i) &= g(x,y,i) - \sum_{j \in S} \pi_{x,y}(j) g(x,y,j) + \sum_{j \in S} p_{x,y}(j|i)V^s(x,y,j).
\end{align*}
Solutions to these Poisson equations do exist (see \cite{borkar1991topics} for example) and are bounded and unique up to an additive constant: For $1 \leq k_1 \leq d_1$ and $1 \leq k \leq d_2$, adding a scalar $c^{k_1}(x,y)$, $c^{k}(x,y)$ to any solution $V_{k_1}^f(x,y,i), V_{k}^s(x,y,i)$ resp.\ for each state $i$ yields another solution in each case. We will consider the solutions $V^f, V^s$ such that $V^f(x,y,i_0) =0, V^s(x,y,i_0)=0, \forall x \in \R^{d_1}, y \in \R^{d_2}$ for a fixed $i_0 \in S$, which renders the solutions unique.  The mappings $(x,y) \mapsto V^f(x,y,i), V^s(x,y,i)$ are Lipschitz uniformly in $i$. (This follows easily from an application of the Cramer's formula in view of our assumptions on the map $(x,y) \mapsto p_{x,y}(\cdot|\cdot)$). Such methods were used by works such as \cite{benveniste_priouret_metivier}.

Now, \eqref{eqn: m_decomp_f} can be decomposed as:
\begin{align*}
	\Delta_n^f(y) &= h(x_n, y, Y_n) - \overline{h}(x_n, y), \\
	&=V^f(x_n, y, Y_n) -  \sum_{j \in S} p_{x_n,y}(j|Y_n)V^f(x_n,y,j) \\
	&= \left ( V^f(x_{n}, y, Y_{n+1}) -  \sum_{j \in S} p_{x_n,y}(j|Y_n)V^f(x_n,y,j) \right )\\&+V^f(x_{n}, y, Y_n) - V^f(x_{n+1}, y, Y_{n+1}) + \left ( V^f(x_{n+1}, y, Y_{n+1}) - V^f(x_n, y, Y_{n+1}) \right )\\
	&=: \zeta_{n}^f(y) +\tau^f_{n}(y)-\tau^f_{n+1}(y)+e^{f,\Delta}_{n}(y),
\end{align*}
where
\begin{align*}
	\zeta_{n}^f(y) &= V^f(x_n, y, Y_{n+1}) -  \sum_{j \in S} p_{x_n,y}(j|Y_n)V^f(x_n,y,j), \\
	\tau^f_{n}(y) &= V^f(x_n, y, Y_n), \\
	e^{f,\Delta}_{n}(y) &= V^f(x_{n+1}, y, Y_{n+1}) - V^f(x_n, y, Y_{n+1}). \numberthis \label{eqn: delta_decomp_fast}
\end{align*}
Similarly, \eqref{eqn: m_decomp_s} can be decomposed as:
\begin{align}
	\Delta_n^s(x) &= \zeta_{n}^s(x) +\tau^s_{n}(x)-\tau^s_{n+1}(x)+e^{s,\Delta}_{n}(x), \label{eqn: delta_decomp_slow}
\end{align}
where
\begin{align*}
	\zeta_{n}^s(x) &= V^s(x, y_n, Y_{n+1}) -  \sum_{j \in S} p_{x,y_n}(j|Y_n)V^s(x,y_n,j), \\
	\tau^s_{n}(x) &= V^s(x, y_n, Y_n), \\
	e^{s,\Delta}_{n}(x) &= V^s(x, y_{n+1}, Y_{n+1}) - V^s(x, y_n, Y_{n+1}).
\end{align*}
Note that $\zeta^f_n(y_n)$ and $\zeta^s_n(x_n)$ are martingale difference sequences. We thus have a two timescale system with Markovian and martingale noise that is reminiscent of the system with martingale noise alone, albeit with some additional error terms. 
\subsubsection{Iterative equations for fluctuations} \label{sec: prelims_iterative_eqns}
In this section, we will construct iterative equations for the fluctuations in each timescale. 

As the limiting value for $\{x(n)\}$ almost surely lies in $\{(\lambda(y), y) : y \in \R^{d_2} \}$, we now use the following device from \cite{pattathil_2018} to write an iterative equation for $\lambda(y_n)$ on the same timescale as $x_n$. This will prove useful in writing down an iterative equation for the fluctuation $\{\beta_n := \lambda(y_n) - x_n \}$. Let $v_n=\lambda\left(y_n\right)$, i.e., $h\left(v_n, y_n \right)=$ $0, n \geq 0$. Also define $\epsilon_n = b(n) / a(n)$. Using Taylor expansion, one can write the following iterative equation for $\left\{v_n\right\}$:
\begin{align*}
	v_{n+1} &= v_n + \nabla \lambda(y_n) \left( y_{n+1} - y_n \right) + e_{n+1}, \\
	&= v_n + \nabla \lambda(y_n) \left( g(x_k, y_k, Y_k) + M^s_{k+1} \right) + e_{n+1} \\
	&= v_n + \nabla \lambda(y_n) \left( \og(x_k, y_k) + \Delta^s_n(x_n) + M^s_{k+1} \right) + e_{n+1} \\
	&= v_n + a(n) \left( \oh(v_n, y_n) +  \epsilon_n\nabla \lambda(y_n) \og(x_n, y_n) + \epsilon_n \nabla \lambda(y_n) \Delta^s_n(x_n) \right.\\&\hspace{15em}\left.+ \epsilon_n \nabla \lambda(y_n)M^s_{n+1} + \frac{e_{n+1}}{a(n)} \right), \numberthis\label{eqn: vn_eqn}
\end{align*}
where 
\begin{align*}
	\frac{	\| e_{n+1}  \|}{a(n)} &= \frac{1}{a(n)}O \left ( \| y_{n+1} - y_n \|^2 \right ) \\
	& \leq \epsilon_n b(n) \newuck \left(B_g^2+2 B_g\left\|M_{n+1}^s\right\|+\left\|M_{n+1}^s\right\|^2\right) \\
	&\leq \newuck \left(\epsilon_n b(n)+\epsilon_n b(n)\left\|M_{n+1}^s\right\|+\epsilon_n b(n)\left\|M_{n+1}^s\right\|^2\right), \numberthis \label{eqn: en_bound}
\end{align*}
as $g(\cdot, \cdot, \cdot)$ is bounded . Subtracting \eqref{eqn: iteration1_decomp} from \eqref{eqn: vn_eqn}:
\begin{align*}
	\beta_{n+1} &= \beta_n + a(n) \left( \oh(v_n, y_n) - \oh(x_n, y_n ) + \epsilon_n\nabla \lambda(y_n) \og(x_n, y_n) + \epsilon_n \nabla \lambda(y_n)M^s_{n+1}  \right.\\&\hspace{10em}\left.+ \epsilon_n \nabla \lambda(y_n)\Delta^s_n(x_n) + \frac{e_{n+1}}{a(n)} - \Delta^f_n(y_n) - M^f_{n+1}\right) \\
	&=\beta_n + a(n) \left( \nabla_x \oh(x_n, y_n ) \beta_n + \epsilon_n\nabla \lambda(y_n) \og(x_n, y_n) + \epsilon_n \nabla \lambda(y_n)M^s_{n+1} \right.\\&\hspace{10em}\left.+ \epsilon_n \nabla \lambda(y_n)\Delta^s_n(x_n) + \frac{e_{n+1}}{a(n)} - \Delta^f_n(y_n) - M^f_{n+1}+ e^{f,T}_n \right), \numberthis\label{eqn: betan_eqn}
\end{align*}
where $e^{f,T}_n = o(\| \beta_n\|)$ is the error resulting from the Taylor expansion of $\oh(\cdot, y_n)$ centered at $x_n$ and evaluated at $v_n$. 

Next, we will write down an iterative equation for the fluctuation on the slower timescale. Using the Taylor expansion for $y^n(\cdot)$ centered at $y^n(t^s(j))$, we have that for $j \geq n$,
\begin{align}
	y^n(t^s(j+1)) = y^n(t^s(j)) + b(j) \left(  \overline{f}(y^n(t^s(j))) -\delta_j \right), \label{eqn: markov_second_ts_disc}
\end{align}
where $\delta_j = O(b(j))$ is the discretization error. Subtracting \eqref{eqn: markov_second_ts_disc} from \eqref{eqn: iteration2_decomp}, we have that
\begin{align}
	\alpha_{j+1}^n = \alpha_j^n + b(j) \left( \nabla \of(y^n(t^s(j))) \alpha_j^n + \delta_j +\Delta^s_n(x_n) +M^s_{n+1}+ e^{s,T}_n + p(x_n, y_n) \right), \label{eqn: alphan}
\end{align}
where $e^{s,T}_n = o(\| \alpha^n_j \| )$ is the error resulting from the Taylor expansion of $\of(\cdot)$ centered at $y^n(t^s(j))$ evaluated at $y_n$. 

\subsection{Broad overview of the proof of Theorem \ref{thm: main_thm}} \label{sec: proof_overview}
In the first part of the proof, we show that the laws of the processes $\{\tilde{w}^n(\cdot), \tilde{u}^n(\cdot), \tilde{y}^n(\cdot), n \geq 0\}$ are tight. 
We shall need the following well-known criterion for tightness of
probability measures on $C(\left[0, T\right]; \mathbb{R}^d)$ due to Kolmogorov:
\begin{lemma}[\cite{billingsley1968convergence}, p.\ 95]\label{lem: Lemma7.4} Let $\{\xi_{\alpha}(\cdot)\}$, for $\alpha$
	belonging to some prescribed index set $J$, be a family of $C(\left[0,
	T\right]; \mathbb{R}^d)$-valued random variables such that the laws of
	$\{\xi_{\alpha}(0)\}$ are tight in $\mathcal{P}(\mathbb{R}^d)$, and for
	some constants $a, b, c > 0$,
	\begin{equation}
		E\left[\|\xi_{\alpha}(t) - \xi_{\alpha}(s)\|^a\right] \leq b|t - s|^{1 + c} \
		\forall \ \alpha \in J, \ t, s \in \left[0,T\right]. \label{eqn: tightness_sufficient}
	\end{equation}
	Then the laws of $\{\xi_{\alpha}(\cdot)\}$ are tight in
	$\mathcal{P}(C(\left[0, T\right]; \mathbb{R}^d))$.
\end{lemma}
Tightness of the laws of $\{\tilde{y}^n(0)\}$ follows from \ref{assumption: bdd_fourth_moment}, as 
\begin{align*}
	P \left( \left \| y^n(0) \right \| > a \right) \leq \frac{E \left[ \left \| y^n(0)\right \|^4 \right]}{a^4} \leq \frac{\newuck}{a^4}.
\end{align*}
The tightness of laws for $\tilde{y}^n(\cdot)$ follows easily from Lemma \ref{lem: Lemma7.4} given the tightness of the laws of their initial conditions established above and the bound
$$E\left[\|\tilde{y}^n(t) - \tilde{y}(s)\|^2\right] \leq K|t-s|^2, \ t > s, |t-s|\leq  T,$$
for some $K > 0$ which follows by a standard application of the Gronwall inequality in view of the Lipschitz conditions of $g,\lambda$. Sections \ref{sec: tightness_fast} and \ref{sec: tightness_slow_martingale} show \eqref{eqn: tightness_sufficient} for the fast and slow timescale respectively, thus proving tightness of the laws of $\{\tilde{u}^n(\cdot)\}$ and $\{\tilde{w}^n(\cdot)\}$ respectively. From Prokhorov's theorem, it then follows that the laws of the processes $\{\tilde{w}^n(\cdot), \tilde{u}^n(\cdot), \tilde{y}^n(\cdot), n \geq 0\}$ are relatively compact in $\mathcal{P} \left( C(\left[0, T\right]; \mathbb{R}^d) \right)^3$ and therefore in  
$\mathcal{P} \left( C(\left[0, T\right]; \mathbb{R}^d)^3 \right)$. 

In the second part of the proof, we characterize the limit of any convergent subsequence in law and show that all convergent subsequences in law have the same limit. For $\{\tilde{y}^n(\cdot)\}$, this trivially follows from Theorem \ref{thm: as_limit} which ensures that it remains bounded with bounded derivatives. Sections \ref{sec: limit_fast_martingale} and \ref{sec: limit_slow_martingale} prove this for $\{\tilde{u}^n(\cdot)\}$ and $\{\tilde{w}^n(\cdot)\}$ respectively.

\subsection{Tightness for the fast timescale} \label{sec: tightness_fast}
Recall \eqref{eqn: betan_eqn} from Section \ref{sec: prelims_iterative_eqns}. For $n \geq 0$:
\begin{align*}
	\beta_{n+1} &=\beta_n + a(n) \left( \nabla_x \oh(x_n, y_n ) \beta_n + \epsilon_n\nabla \lambda(y_n) \og(x_n, y_n) + \epsilon_n \nabla \lambda(y_n)M^s_{n+1} \right.\\&\hspace{10em}\left.+ \epsilon_n \nabla \lambda(y_n)\Delta^s_n(x_n) + \frac{e_{n+1}}{a(n)} - \Delta^f_n(y_n) - M^f_{n+1}+ e^{f,T}_n \right). 
\end{align*}
Using this equation repeatedly, we get that for any $i \geq 1$,
\begin{align*}	\beta_{n+i} &= \prod_{j=n}^{n+i-1} \left( I+a(j) \nabla_x \oh(x_j, y_j) \right) \beta_n \\&+ \sum_{j=n}^{n+i-1} a(j) \left (\prod_{k=j+1}^{n+i-1} \left( I+a(k) \nabla_x \oh(x_k, y_k)  \right) \right ) \Biggl(  \epsilon_j\nabla \lambda(y_j) \og(x_j, y_j) \\&\hspace{2em}+ \epsilon_j \nabla \lambda(y_j)M^s_{j+1} + \epsilon_j \nabla \lambda(y_j)\Delta^s_j(x_j) + \frac{e_{j+1}}{a(j)} - \Delta^f_j(y_j) - M^f_{j+1}+ e^{f,T}_j \Biggr).
\end{align*}
Let $\{ \beta_{n+i}^n\}_{i=0}^{m^f(n)-n}$ be a sequence satisfying the above iteration with $\beta_n^n=0$. Then,
\begin{align*}
	\beta_{n+i}^n &=   \sum_{j=n}^{n+i-1} a(j) \left (\prod_{k=j+1}^{n+i-1} \left( I+a(k) \nabla_x \oh(x_k, y_k)  \right) \right ) \Biggl(  \epsilon_j\nabla \lambda(y_j) \og(x_j, y_j) \\&\hspace{2em}+ \epsilon_j \nabla \lambda(y_j)M^s_{j+1} + \epsilon_j \nabla \lambda(y_j)\Delta^s_j(x_j) + \frac{e_{j+1}}{a(j)} - \Delta^f_j(y_j) - M^f_{j+1}+ e^{f,T}_j \Biggr).
\end{align*}
Recall the following definitions from Section \ref{sec: notation}:
\begin{align*}
	u_k &:= \frac{\beta_k}{\sqrt{a(k)}}, \\
	\nu_k^n &:= \frac{\beta_k^n}{\sqrt{a(k)}}.
\end{align*}
Then,
\begin{align*}
	\nu_{n+i}^n &=  \sum_{j=n}^{n+i-1} \sqrt{a(j)} \left (\prod_{k=j+1}^{n+i-1} \left( I+a(k) \nabla_x \oh(x_k, y_k)  \right) \sqrt{\frac{a(k)}{a(k+1)}} \right )  \Biggl(  \epsilon_j\nabla \lambda(y_j) \og(x_j, y_j) \\&\hspace{4em}+ \epsilon_j \nabla \lambda(y_j)M^s_{j+1} + \epsilon_j \nabla \lambda(y_j)\Delta^s_j(x_j) + \frac{e_{j+1}}{a(j)} \\&\hspace{14em}- \Delta^f_j(y_j) - M^f_{j+1}+ e^{f,T}_j \Biggr) \sqrt{\frac{a(j)}{a(j+1)}}. \numberthis \label{eqn: first_scaled_iterate}
\end{align*}
As defined in Section \ref{sec: notation},  $\nu^n(t), t \in \left[t^f(n), t^f(n)+T\right]$ is a piecewise linear interpolation of $\{\nu^n_k\}_{k=n}^{m^f(n)}$ on the algorithmic timescale. Define $\mu^n_j, n \leq j \leq m^f(n),$ such that $\mu^n_j = u_j - \nu^n_j$. Then
\begin{align*}
	\mu^n_j = \prod_{j=n}^{j-1} \left( I+a(j) \nabla_x \oh(x_j, y_j) \right) \sqrt{\frac{a(j)}{a(j+1)}} u_n.
\end{align*}
For $t \in \left[t^f(n), t^f(n)+T\right]$, let $\mu^n(t) = u^n(t) - \nu^n(t)$, which is a linearly interpolated version of $\{\mu^n_j, n \leq j \leq m^f(n)\}$. Thus, it is enough to show tightness for the laws of $\{\mu^n(\cdot)\}$ and $\{\nu^n(\cdot)\}$ to prove tightness for the laws of $\{u^n(\cdot)\}$. 
\subsubsection{Proving tightness for $\{\mu^n(\cdot)\}_{n \geq 1}$ } \label{sec: martingale_tightness_fast_mu}
To show tightness for the laws of $\{\mu^n(\cdot)\}$, we will use Lemma \ref{lem: Lemma7.4}. We first need the following intermediate result: for $n \leq  j \leq m^f(n)$,
\begin{align*}
	\prod_{r=n}^{j-1} \sqrt{\frac{a(r)}{a(r+1)}} &= \exp \left(\sum_{r=n}^{j-1} \ln \left( \sqrt{\frac{a(r)}{a(r+1)} }  \right) \right ) \\
	&\leq \exp \left(\frac{1}{2}\left( \sum_{r=n}^{j-1} \frac{a(r)}{a(r+1)} - 1\right) \right) \\
	&= \exp \left(\frac{1}{2} \sum_{r=n}^{j-1} a(r) \left(\frac{1}{a(r+1)} - \frac{1}{a(r)}\right) \right ) \\
	&\leq \exp \left(\newuck \sum_{r=n}^{j-1} a(r) \right ),
\end{align*}
where the last step follows from \ref{assumption: alpha_beta}. 
Note that $\mu^n(0) = u_n$, and for $n \leq k < j \leq m^f(n)$, it follows that
\begin{align*}
	&E \left[\left \| \mu^n_j - \mu^n_k \right \|^4  \right] \\
	&= E \left[ \left \| u_n \right \|^4 \left \| \prod_{r=n}^{j-1} \left (1 + a(r)\nabla_x \oh(x_{r}, y_{r}) \right ) \sqrt{\frac{a(r)}{a(r+1)}} \right.\right. \\ &\hspace{15em} \left.\left.- \prod_{r=n}^{k-1} \left (1 + a(r)\nabla_x \oh(x_{r}, y_{r})   \right ) \sqrt{\frac{a(r)}{a(r+1)}} \right \|^4  \right] \\
	&= E \left[ \left \| u_n \right \|^4 \left \| \exp \left( \ln \left (\prod_{r=n}^{j-1} \left (1 + a(r)\nabla_x \oh(x_{r}, y_{r}) \right )  \sqrt{\frac{a(r)}{a(r+1)}} \right ) \right ) \right.\right. \\ &\hspace{15em} \left.\left.- \exp \left( \ln \left (\prod_{r=n}^{k-1} \left (1 + a(r)\nabla_x \oh(x_{r}, y_{r})  \right ) \sqrt{\frac{a(r)}{a(r+1)}} \right )  \right )  \right \|^4  \right] \\
	&\leq  E \left[ \left \| u_n \ \exp \left(\ln \left (\prod_{r=n}^{m^f(n)-1} \left (1 + a(r)\nabla_x \oh(x_{r}, y_{r}) \right ) \sqrt{\frac{a(r)}{a(r+1)}}  \right )\right) \right \|^4 \times \right. \\ & \left. \hspace{2em} \left \| {\ln \left (\prod_{r=n}^{j-1} \left (1 + a(r)\nabla_x \oh(x_{r}, y_{r}) \right ) \sqrt{\frac{a(r)}{a(r+1)}} \right )} \right.\right. \\ &\hspace{15em} \left.\left.- { \ln \left (\prod_{r=n}^{k-1} \left (1 + a(r)\nabla_x \oh(x_{r}, y_{r}) \right ) \sqrt{\frac{a(r)}{a(r+1)}} \right ) } \right \|^4  \right],
\end{align*}
where the last step follows from the mean value theorem.
Thus,
$$
\begin{aligned}
	&E \left[\left \| \mu^n_j - \mu^n_k \right \|^4  \right] \\
	&\leq  \exp \left( 4\newuck\label{const: D_xh} \sum_{r=n}^{m^f(n)-1} a(r) \right ) \  E \left[ \left \| u_n \right \|^4 \left \|  {\ln \left (\prod_{r=k}^{j-1} \left (1 + a(r)\nabla_x \oh(x_{r}, y_{r}) \right ) \sqrt{\frac{a(r)}{a(r+1)}} \right )} \right \|^4 \right] \\
	&\leq \newuck \exp \left( 4\olduck{const: D_xh}\sum_{r=n}^{m^f(n)-1} a(r) \right ) \  E \left[ \left \| u_n \right \|^4 \left \|  { \sum_{r=k}^{j-1}  a(r)  } \right \|^4 \right] \\
	&\leq \newuck E \left[ \left \| u_n \right \|^4 \right] \left(t^f(j) - t^f(k)\right)^4,
\end{aligned}
$$
for a suitable bound $\olduck{const: D_xh}> 0$ on $\|\nabla_x \oh(\cdot,\cdot)\|$. This follows from repeated applications of the result that $\ln(1+x) \leq x$ and \ref{assumption: hg}.

To complete the proof of tightness for the laws of $\{\mu^n(\cdot)\}$, it thus suffices to show that the fourth moments of $\{u_n\}_{n \geq 1}$ are uniformly bounded, i.e., $\{E \left [ \left \| u_n \right \|^4 \right ]\}_{n \geq 1}$ is a bounded sequence. 

\subsubsection{Proof that $\left\{ E \left[\left \| u_n \right \|^4 \right] \right\}_{n \geq 1} $ is a bounded sequence.} \label{sec: un_bound}

Let $x^n(t), t \geq t^f(n),$ be the unique solution of $\dot{x}(t) = h(x(t), y_n)$ starting at $t^f(n)$: 
\begin{align*}
	\dot{x}^n(t) = \oh(x^n(t), y_n), t \geq t(n); x^n(t(n)) = x_n. 
\end{align*}
Define $\gamma^n_j = \frac{x_j - x^n(t^f(j))}{\sqrt{a(j)}}, n \leq j \leq m^f(n)$. The following lemma shows the boundedness of the fourth moments thereof, the proof of which can be found in Appendix \ref{sec: appendix_martingale_gamma}. 
%
\begin{lemma} \label{lem: first_timescale_with_ode_fclt}
	There exists a constant $\kappa$ such that for each $ n$,
	\begin{align*}
		\underset{n \leq k \leq m^f(n)}{\sup} E \left[\| \gamma^n_k \|^4 \right] < \kappa < \infty.
	\end{align*}
\end{lemma}

Now, to show that the fourth moments of $\{u_n\}$ are bounded using the above result, we have the following equation:
\begin{align*}
	\gamma_n^{m^{f-}(n)} &= \frac{x_{n} - x^{m^{f-}(n)}(t^f(n))}{\sqrt{a(n)}}\\
	&= u_n + \frac{\lambda(y_n)-x^{m^{f-}(n)}(t^f(n))}{\sqrt{a(n)}}.
\end{align*}
By rearranging the above equation, we have that
\begin{align*}
	\left \| \gamma_n^{m^{f-}(n)} - u_n \right \| &= \frac{ \left \| \lambda(y_{n})-x^{m^{f-}(n)}(t^f(n)) \right \| }{\sqrt{a(n)}} \\
	&\leq \frac{\left \| \lambda(y_n) - \lambda(y_{m^{f-}(n)})  \right \| + \left \| \lambda(y_{m^{f-}(n)}) -x^{m^{f-}(n)}(t^f(n)) \right \|}{\sqrt{a(n)}}  \\
	&\leq \frac{\left \| \lambda(y_n) - \lambda(y_{m^{f-}(n)})  \right \|}{\sqrt{a(n)}} + c e^{-dT} \frac{\left \| \lambda(y_{m^{f-}(n)}) - x_{m^{f-}(n)} \right \|}{\sqrt{a(n)}}  \\
	&= \sqrt{\frac{a(m^{f-}(n))}{a(n)}}\frac{\left \| \lambda(y_n) - \lambda(y_{m^{f-}(n)}) \right \|}{\sqrt{a(m^{f-}(n))}} + c e^{-dT} \sqrt{\frac{a(m^{f-}(n))}{a(n)}} u_{m^{f-}(n)},  \numberthis \label{eqn: first_ts_clt_bounding}
\end{align*}
where the penultimate step follows from \ref{assumption: fast_ode_convergence_martingale}. The fourth moment of the first term can be shown to converge to 0 uniformly in $T$ as $n \uparrow \infty$ as follows:
\begin{align*}
	&E \left [  \left ( \sqrt{\frac{a(m^{f-}(n))}{a(n)}} \frac{\left \| \lambda(y_n) - \lambda(y_{m^{f-}(n)}) \right \|}{\sqrt{a(m^{f-}(n))}} \right )^4 \right ] \\
	&\leq 	\newuck \ E \left [ \left (  \frac{\sum^{n-1}_{k=m^{f-}(n)} \frac{b(k)}{a(k)} a(k) \left( \left \| g(x_k, y_k,Y_k) \right \| + \left \| M^s(k+1) \right \| \right) }{\sqrt{a(m^{f-}(n))}} \right )^4 \right ] \\
	&\leq	\newuck \ T^3 \ E \left [   \frac{ \sum^{n-1}_{k=m^{f-}(n)} \left (\frac{b(k)}{a(k)} \right )^4 a(k) \left( \left \| g(x_k, y_k,Y_k) \right \|^4 + \left \| M^s(k+1) \right \|^4 \right) }{a(m^{f-}(n))^2} \right ] \\ 
	&\leq \newuck \ T^3 \underset{m^{f-}(n) \leq k \leq n-1}{\sup } \left (\frac{b(k)}{ a(k)} \right )^4  E \left [  \frac{ \sum^{n-1}_{k=m^{f-}(n)} a(k) }{a(m^{f-}(n))^2}  \right ] \\
	&\leq \frac{\newuck \ T^4}{a(m^{f-}(n))^2}  \underset{m^{f-}(n) \leq k \leq n-1}{\sup } \left (\frac{b(k)}{ a(k)} \right )^4 \\
	&\leq  \newuck \ T^4 \underset{m^{f-}(n) \leq k \leq n-1}{\sup } \left (\frac{b(k)}{ a(k)^{3/2}} \right )^4   \overset{n \uparrow \infty}{\rightarrow} 0, \numberthis\label{eqn: first_ts_CLT_y_error}
\end{align*}
where the first step follows from the Lipschitz property of $\lambda(\cdot)$ and \ref{assumption: LemmaA.28}, the second step follows from using the convexity of the map $x \mapsto x^4$ map and that $(a+b)^4 \leq 8(a^4+b^4)$, the third step follows from the boundedness of $g(\cdot, \cdot,\cdot)$ and the fourth moment of $M^s(\cdot)$, the fifth step follows from $a(n)$ decreasing to zero monotonically, and the last step follows from $b(n) = o(a(n)^{3/2})$. It follows that for some $\newuck\label{const: firsttsCLTyerrorub}>0$ and for a large enough $n$,
$$
\begin{aligned}
	\underset{n \leq k \leq m^f(n)}{\max} E\left [ \left (\frac{\left \| \lambda(y_k) - \lambda(y_{m^{f-}(k)}) \right \|}{\sqrt{a(k)}} \right )^4 \right] < \olduck{const: firsttsCLTyerrorub}.
\end{aligned}
$$
The final step is to show that $\{E[ \|u_n\|^4 ]\}_{n \geq 1}$ is bounded. 

From \eqref{eqn: first_ts_clt_bounding}, we have that
\begin{align*}
	\|u_n\| \leq \| \gamma_n^{m^{f-}(n)}  \| + \frac{\left \| \lambda(y_n) - \lambda(y_{m^{f-}(n)}) \right \|}{\sqrt{a(n)}} + c e^{-dT} \sqrt{\frac{a(m^{f-}(n))}{a(n)}} \left \| u_{m^{f-}(n)} \right \|.
\end{align*}
From \ref{assumption: LemmaA.28}, it also follows that we can choose a $T$ large enough so that for all $k \in \mathbb{N}$, 
\begin{align*}
	c e^{-dT} \sqrt{\frac{a(k)}{a(m^f(k))}} < \zeta(T) < \frac{1}{2^{5/4}},
\end{align*}
where $\zeta(T)$ is a decreasing function of $T$. Let $b^u_n := \underset{n \leq k \leq m^f(n)}{\max} E \left[ \| u_k \|^4 \right]$. Combining the above results, boundedness of the fourth moment of $E[\| \gamma_k^{m^{f-}(k)}  \|^4 ] $ from Lemma \ref{lem: first_timescale_with_ode_fclt}, and the fact that $(a+b+c)^4 \leq 16(a^4+b^4+c^4)$, we have that
\begin{align*}
	b^u_n &\leq 16\left( \underset{n \leq k \leq m^f(n)}{\max} E[\| \gamma_k^{m^{f-}(k)}  \|^4 ]  + \olduck{const: firsttsCLTyerrorub}  + \zeta(T)^4 \underset{m^{f-}(n) \leq k \leq n}{\max} E \left [\left \| u_{k} \right \|^4 \right ] \right) \\
	&\leq \newuck + 16\zeta(T)^4 b^u_{m^{f-}(n)} \leq \newuck + \frac{1}{2}b^u_{m^{f-}(n)}.
\end{align*}
It follows that $\{E[\| u_n \|^4 ]\}_{n \geq 1}$ is bounded.

\subsubsection{ Proving tightness for $\{\nu^n(\cdot)\}_{n \geq 1}$.} \label{sec: tightness_nu}
The following lemma coupled with Lemma \ref{lem: Lemma7.4} proves tightness for $\{\nu^n(\cdot)\}_{n \geq 1}$. 
\begin{lemma} \label{lem: tightness_first_ts}
	For $0 \leq n \leq k < \ell \leq m^f(n)$,
	\begin{align*}
		E \left[ \left \| \nu_{\ell} - \nu_{k} \right \|^4 \right] =O \left( \left( t^f(\ell) - t^f(k)   \right)^2 \right). 
	\end{align*}
\end{lemma}
\begin{proof}
Recall that for $ 0 \leq i \leq m^f(n)-n$,
\begin{align*}
	\nu_{n+i}^n &=  \sum_{j=n}^{n+i-1} \sqrt{a(j)} \left (\prod_{k=j+1}^{n+i-1} \left( I+a(k) \nabla_x \oh(x_k, y_k)  \right) \sqrt{\frac{a(k)}{a(k+1)}} \right )\\& \hspace{4em} \Biggl( e^{f,T}_j + \epsilon_j\nabla \lambda(y_j) \og(x_j, y_j) - M^f_{j+1} + \epsilon_j \nabla \lambda(y_j)M^s_{j+1} + \frac{e_{j+1}}{a(j)} \\&\hspace{14em} + \epsilon_j \nabla \lambda(y_j)\Delta^s_j(x_j) - \Delta^f_j(y_j)   \Biggr) \sqrt{\frac{a(j)}{a(j+1)}}. \numberthis \label{eqn: vn}
\end{align*}
By \ref{assumption: alpha_beta}, $\sqrt{a(j)/a(j+1)}$ is uniformly bounded in $j$. Also, as $\oh$ is uniformly Lipschitz, $\nabla_x \oh(\cdot,\cdot)$ is uniformly bounded. As seen previously, for $n \leq k <\ell< m^f(n)$,
$$
\begin{aligned}
	\left\|\prod_{r=k+1}^{\ell} \left (1 + a(r)\nabla_x \oh(x_{r}, y_{r}) \right )\right\| &\leq \exp \left(\olduck{const: D_xh}\sum_{r=k+1}^{\ell} a(r) \right ) \\
	&\leq \exp \left(\olduck{const: D_xh}(T+1) \right ), 
\end{aligned}
$$
for a suitable bound $\olduck{const: D_xh}> 0$ on $\|\nabla_x \oh(\cdot,\cdot)\|$. We define
\begin{align*}
	\Gamma_{n,j}^{f,k} 	:= \prod_{r=j}^{k} \left(I+a(r) \nabla \oh(x_r, y_r)\right) \sqrt{\frac{a(r)}{a(r+1)}} .
\end{align*}
Recall that this quantity is bounded. Also, for some $\newuck\label{const: efT_fast}>0$,
$$
\begin{aligned}
	\| e^{f,T}_j \| \leq \olduck{const: efT_fast} \| \beta_j \|  \leq \sqrt{a(j)} \olduck{const: efT_fast} \| u_j \|,
\end{aligned}
$$
for a large enough $j$. Thus for large enough $j$,
\begin{align*}
	E \left[	\left\| 	\sum_{j=k+1}^{\ell} \sqrt{a(j)} 	\Gamma^{f,\ell}_{n,j+1} e^{f,T}_j \sqrt{\frac{a(j)}{a(j+1)}} \right\| ^4 \right] &\leq \newuck E \left [ \left( \sum_{j=k+1}^{\ell} a(j) \| u_j \|  \right)^4 \right ] \\
	&\leq \left(\sum_{j=k+1}^{\ell} a(j) \right)^3 \left(\sum_{j=k+1}^{\ell} a(j) E \left [\| u_j \|^4 \right ] \right) \\
	&\leq \left(\sum_{j=k+1}^{\ell} a(j) \right)^2,
\end{align*}
using Jensen's inequality and the result that the fourth moment of the sequence $\{u_n\}_{n \geq 1}$ is bounded. 

As $\lambda(\cdot)$ is uniformly Lipschitz, $\nabla \lambda (\cdot) $ is uniformly bounded. Thus, for some $\newuck\label{const: eta_j}>0$,
\begin{align*}
	\epsilon_j \og\left(x_j, y_j\right) \nabla \lambda\left(y_j \right ) &\leq \olduck{const: eta_j} \epsilon_j.
\end{align*}
Thus, 
$$
\begin{aligned}
	&\left\| 	\sum_{j=k+1}^{\ell} \sqrt{a(j)} \Gamma_{n,j+1}^{f,\ell}   \epsilon_j \og\left(x_j, y_j\right) \nabla \lambda\left(y_j \right ) \sqrt{\frac{a(j)}{a(j+1)}} \right\| \\
	&\leq \olduck{const: eta_j} e^{\olduck{const: D_xh}(T+1)}  \sum_{j=k+1}^{\ell} \sqrt{a(j)} \frac{b(j)}{a(j)} \leq \olduck{const: eta_j} e^{\olduck{const: D_xh}(T+1)} \sum_{j=k+1}^{\ell} a(j),
\end{aligned}
$$
by the assumption $b(j) \leq a(j)^{3/2}$. 

We need the following consequence of the Burkholder inequality.
\begin{lemma}[\cite{Borkar2008}, Lemma 7.1] \label{lem: lemma7.1}
	Let $\{X_n\}$ be a zero mean martingale w.r.t.\ the increasing
	$\sigma$-fields $\{\mathcal{F}_n\}$ with $X_0 = 0$ (say) and
	$\sup_{n\leq N}E\left[|X_n|^4\right] < \infty$. Let $Y_n
	\defn X_n - X_{n-1}, n \geq 1$. For a suitable constant $K > 0$,
	\begin{displaymath}
		E\left[\sup_{n\leq N}|X_n|^4\right] \leq K \left (\sum_{m=1}^NE\left[Y_m^4\right]  +
		E\left[(\sum_{m=1}^NE\left[Y_m^2|\mathcal{F}_{m-1}\right])^2\right] \right ).
	\end{displaymath}
\end{lemma} 

From \eqref{eqn: vn}, define
\begin{align*}
	\Psi_{k,m}^{(1)} = -\sum_{j=k+1}^{m} \sqrt{a(j)} \Gamma^{f,m}_{n,j+1} M^f_{j+1} \sqrt{\frac{a(j)}{a(j+1)}}. 
\end{align*}
Then by Lemma \ref{lem: lemma7.1}, we have that
\begin{align*}
	E \left[ \sup_{k+1 \leq m < \ell} \| \Psi^{(1)}_{k,m} \|^4 \right] \leq \newuck \left( \left( \sum_{j=k+1}^{\ell} a(j) \right)^2 + \sum_{j=k+1}^{\ell} a(j)^2 \right).
\end{align*}
Likewise, define 
\begin{align*}
	\Psi_{k,m}^{(2)} = \sum_{j=k+1}^{m} \sqrt{a(j)} \Gamma^{f,m}_{n,j+1} \epsilon_j \nabla \lambda (y_j) M_{j+1}^s \sqrt{\frac{a(j)}{a(j+1)}}. 
\end{align*}
By using $\epsilon_j \leq 1$ and the fact that $\nabla \lambda (y_j)$ is uniformly bounded, by Lemma \ref{lem: lemma7.1}, we have that,
\begin{align*}
	E \left[ \sup_{k+1 \leq m < \ell} \| \Psi_{k,m}^{(2)} \|^4 \right] \leq \newuck \left( \left( \sum_{j=k+1}^{\ell} a(j) \right)^2 + \sum_{j=k+1}^{\ell} a(j)^2 \right).
\end{align*}
Recall from \eqref{eqn: en_bound} that
\begin{align*}
	\left\| \frac{e_{k+1}}{a(k)}\right\|  &\leq \newuck \left(\epsilon_k b(k)+\epsilon_k b(k)\left\|M_{k+1}^s\right\|+\epsilon_k b(k)\left\|M_{k+1}^s\right\|^2\right).
\end{align*}
Thus,
\begin{align*}
	&E \left [ \left \| \sum_{j=k+1}^{\ell} \sqrt{a(j)} \Gamma^{f,m}_{n,j+1} \frac{e_{j+1}}{a(j)}\sqrt{\frac{a(j)}{a(j+1)}} \right \|^4 \right ] \\
	&= E \left [ \left \| \sum_{j=k+1}^{\ell} a(j) \Gamma^{f,m}_{n,j+1} \frac{e_{j+1}}{a(j)^{3/2}}\sqrt{\frac{a(j)}{a(j+1)}} \right \|^4 \right ] \\
	&\overset{(a)}{\leq} \left(\sum_{j=k+1}^{\ell} a(j)\right)^4 E \left [  \sum_{j=k+1}^{\ell} \frac{a(j)}{\left(\sum_{j=k+1}^{\ell} a(j)\right)} \left \|  \Gamma^{f,m}_{n,j+1} \frac{e_{j+1}}{a(j)^{3/2}} \sqrt{\frac{a(j)}{a(j+1)}}  \right \|^4  \right ] \\
	&= \left(\sum_{j=k+1}^{\ell} a(j)\right)^3 E \left [  \sum_{j=k+1}^{\ell} \frac{1}{a(j) } \left \|  \Gamma^{f,m}_{n,j+1} \frac{e_{j+1}}{a(j)} \sqrt{\frac{a(j)}{a(j+1)}}  \right \|^4  \right ] \\
	&\overset{(b)}{\leq} 27 \newuck T^3 E \left [  \sum_{j=k+1}^{\ell} b(j)^3 \epsilon_j^5  \right ] \\
	&\overset{(c)}{\leq} \newuck \left(\sum_{j=k+1}^{\ell} b(j)\right)^3 \\
	&\overset{(d)}{\leq} \newuck \left(\sum_{j=k+1}^{\ell} a(j)\right)^3,
\end{align*}
where $(a)$ follows from Jensen's inequality, $(b)$ follows from \ref{assumption: martingale} and the fact that $(a+b+c)^p \leq 3^{p-1}(a^p + b^p + c^p), p \geq 1$, $(c)$ follows from $\epsilon_j \leq 1$ and $\sum_{j=k+1}^{\ell} b(j)^3 \leq \left(\sum_{j=k+1}^{\ell} b(j)\right)^3$, and $(d)$ follows from $b(j) \leq a(j)$. 

As $\Delta^s_j(x_j) = g(x_j, y_j, Y_j) - \og(x_j, y_j)$, it follows from the boundedness of $g(\cdot,\cdot,\cdot)$ and Jensen's inequality that 
\begin{align*}
	&E \left [ \left \| \sum_{j=k+1}^{\ell} \sqrt{a(j)} \Gamma_{n,j+1}^{f,\ell} \epsilon_j \nabla \lambda(y_j) \Delta^s_j(x_j) \sqrt{\frac{a(j)}{a(j+1)}} \right \|^4 \right ] \\
	&\leq E \left [ \left ( \sum_{j=k+1}^{\ell} \sqrt{a(j)} \Gamma_{n,j+1}^{f,\ell} \epsilon_j \nabla \lambda(y_j) \left \|  \Delta^s_j(x_j)\right \| \sqrt{\frac{a(j)}{a(j+1)}} \right )^4 \right ] \\
	&\leq \newuck  E \left [ \left ( \sum_{j=k+1}^{\ell} \sqrt{a(j)} \epsilon_j \right )^4 \right ] \\
	&\leq \newuck \left(\sum_{j=k+1}^{\ell} a(j)\right)^4.
\end{align*}
Here, in the penultimate step, we used the assumption that $a(j)^{3/2} > b(j)$. 

The term with $\Delta^f_j(y_j)$ is the last one left to handle. Recall the decomposition for $\Delta^f_j(y_j)$ shown in \eqref{eqn: delta_decomp_fast}:
\begin{align}
	\Delta_j^f(y_j) &= \zeta_{j}^f(y_j) +\tau^f_{j}(y_j)-\tau^f_{j+1}(y_j)+e^{f,\Delta}_{j}(y_j).
\end{align}
We will now handle each of these terms. Recall that $\zeta^f_j(y_j)$ is a bounded martingale difference sequence. Using Lemma \ref{lem: lemma7.1} for 
\begin{align*}
	\Psi_{k,m}^{(3)} := \sum_{j=k+1}^{m} \sqrt{a(j)} \Gamma_{n,j+1}^{f,m} \zeta^{s}_{j}(y_j) \sqrt{\frac{a(j)}{a(j+1)}},
\end{align*}
we have that 
\begin{align*}
	E \left [ \underset{k+1 \leq m < \ell}{\sup} \| \Psi_{k,m} \|^4 \right ] &\leq \left(\sum_{j=k+1}^{\ell} a(j)\right)^2 + \sum_{j=k+1}^{\ell} a(j)^2 \\
	&\leq \newuck \left(\sum_{j=k+1}^{\ell} a(j)\right)^2 .
\end{align*}
Using a summation by parts argument, we have that
\begin{align*}
	&\sum_{j=k+1}^{\ell} \sqrt{a(j)} \Gamma_{n,j+1}^{f,\ell} \left(\tau^f_{j+1}(x_j) -  \tau_{j}^f(x_j)   \right) \sqrt{\frac{a(j)}{a(j+1)}} \\
	&= \sqrt{a(\ell)} \tau^f_{\ell+1}(x_{\ell}) \sqrt{\frac{a(\ell)}{a(\ell+1)}} - \sqrt{a(k+1)}\Gamma_{n,k+2}^{f,\ell} \tau^f_{k+1}(x_{k+1}) \sqrt{\frac{a(k+1)}{a(k+2)}} \\
	&+\sum_{j=k+2}^{\ell} \Gamma_{n,j+1}^{f,\ell} \tau_j^f(x_j) \sqrt{a(j)} a(j-1) \left( \nabla \of(y^n(t^f(j))) + \frac{1}{a(j)} - \frac{1}{a(j-1)} \right). \numberthis \label{eqn: sum_of_parts_fast}
\end{align*}
Thus,
$$
\begin{aligned}
	&\left \| \sum_{j=k+1}^{\ell} \sqrt{a(j)} \Gamma_{n,j+1}^{f,\ell} \left( \tau_{j}^f(x_j)  - \tau^f_{j+1}(x_j) \right) \sqrt{\frac{a(j)}{a(j+1)}} \right \| \\
	&\leq \newuck \left( \left \| \sqrt{a(\ell)} \tau^f_{\ell+1}(x_{\ell}) \right \|+ \left \|\sqrt{a(k+1)}\Gamma_{n,j+1}^{f,\ell} \tau^f_{k+1}(x_{k+1}) \right \| \right.\\
	&\left.+\left \| \sum_{j=k_1+2}^{\ell} \Gamma_{n,j+1}^{f,\ell} \tau_j^f(x_j) \sqrt{a(j)} a(j-1) \left( \nabla \of(y^n(t^f(j))) + \left (\frac{1}{a(j)} - \frac{1}{a(j-1)} \right )I \right) \right \| \right). 
\end{aligned}
$$
Using $(a+b+c)^4 \leq 16(a^4+b^4+c^4)$, \ref{assumption: alpha_beta},  and the fact that $V^f(\cdot, \cdot, \cdot)$ is bounded, 
\begin{align*}
	&E \left [ \left \| \sum_{j=k+1}^{\ell} \sqrt{a(j)} \Gamma_{n,j+1}^{f,\ell} \left( \tau_{j}^f(x_j)  - \tau^f_{j+1}(x_j) \right) \sqrt{\frac{a(j)}{a(j+1)}} \right \|^4 \right ] \\
	&\leq \newuck \left( a(\ell)^2 + a(k+1)^2 + \left(\sum_{j=k+2}^{\ell} \sqrt{a(j)}a(j-1)  \right)^4  \right) \\
	&\leq \newuck \left( \left(\sum_{j=k+1}^{\ell} a(j)\right)^2 +\left(\sum_{j=k+1}^{\ell} a(j)\right)^2 \left(\sum_{j=k+2}^{\ell} a(j-1)^2  \right)^2 \right) \\
	&\leq \newuck \left(\sum_{j=k+1}^{\ell} a(j)\right)^2.
\end{align*}
Here, the second last inequality follows from the Cauchy-Schwartz inequality and the last inequality follows from 
\begin{align*}
	\left(\sum_{j=k+2}^{\ell} a(j-1)^2  \right)^2  \leq \left(\sum_{j=k+2}^{\ell} a(j-1)  \right)^4 \leq T^4. 
\end{align*} 

Finally, as $V^f(\cdot, y, Y_{n+1})$ is Lipschitz,
$$
\begin{aligned}
	\left \| e^{f,\Delta}_{j}(y) \right \| &= \left \| V^f(x_{j+1}, y, Y_{j+1}) - V^f(x_j, y, Y_{j+1}) \right \| \\
	&\leq \newuck \left \| x_{j+1} - x_j \right \| \\
	&\leq\newuck a(j) \left \| h(x_j, y_j, Y_j) + M^f_{j+1}\right \|,
\end{aligned}
$$
Thus, by Jensen's inequality, the boundedness of $h(\cdot, \cdot, \cdot)$ and the fourth moments of $M^f_{j+1}$, we have that
\begin{align*}
	E \left [ \left \| \sum_{j=k+1}^{\ell} \sqrt{a(j)} \Gamma_{n,j+1}^{f,\ell} e^{f,\Delta}_{j}(x_j) \sqrt{\frac{a(j)}{a(j+1)}} \right \|^4 \right ] 
	&\leq \newuck T^3 \left(\sum_{j=k+1}^{\ell} a(j)^{3} \right)^4 \\
	&\leq \newuck \left(\sum_{j=k+1}^{\ell} a(j)\right)^3.
\end{align*}
Combining the above upper bounds for each term in \eqref{eqn: vn}, we have
$$
\begin{aligned}
	E \left[ \left \| \nu_{k} \right \|^4 \right] &\leq \newuck\label{const: first_nu_gronwall}\left( \sum_{j=n}^{k} a(j) \right)^2 ,
\end{aligned}
$$
for some constant $\olduck{const: first_nu_gronwall}$ which can depend on $T$.
Thus,
\begin{align*}
	E \left[ \left \| \nu_{\ell} - \nu_{k} \right \|^4 \right] 	&\leq \newuck \left( \sum_{j=k+1}^{\ell} a(j) \right)^2 \\
	&= O \left( \left( t^f(\ell) - t^f(k)   \right)^2 \right). 
\end{align*}
\end{proof}
Tightness follows from Lemma \ref{lem: Lemma7.4}.	 
\subsection{Tightness for the slow timescale} \label{sec: tightness_slow_martingale}
For $j \geq n$, recall \eqref{eqn: alphan}:
\begin{align}
	\alpha_{j+1}^n = \alpha_j^n + b(j) \left( \nabla \of(y^n(t^s(j))) \alpha_j^n + \delta_j +\Delta^s_j(x_j) +M^s_{j+1}+ e^{s,T}_j + p(x_j, y_j) \right), \label{eqn: alphan_1}
\end{align}
For $0 \leq i \leq m^s(n)-n$, as $\alpha^n_n=0$ by definition,
\begin{align*}
	\alpha^n_{n+i} &= \sum_{j=n}^{n+i-1} b(j) \prod_{k=j+1}^{n+i-1} \left( I+ b(k) \nabla \of( y^n(t^s(k)) ) \right) \left(\delta_j +\Delta^s_j(x_j) \right.\\&\hspace{18em}\left.+M^s_{j+1}+ e^{s,T}_j + p(x_j, y_j) \right), \\
	w^n_{n+i} &= \sum_{j=n}^{n+i-1} \frac{b(j)}{\sqrt{a(j)}} \prod_{k=j+1}^{n+i-1} \left( I+ b(k) \nabla f( y^n(t^s(k))) \right) \left( \delta_j +\Delta^s_j(x_j) \right.\\&\hspace{8em}\left.+M^s_{j+1}+ e^{s,T}_j + p(x_j, y_j) \right)  \sqrt{\frac{a(k)}{a(k+1)}} \sqrt{\frac{a(j)}{a(j+1)}} .
\end{align*}
As defined in Section \ref{sec: notation}, $w^n(t), t \in [t^s(n), t^s(n) + T],$ is such that $w^n(t^s(j))  = w^n_j,$ and linearly interpolated on each interval $[t^s(j),t^s(j+1)]$ for $n \leq j \leq m^s(n)$. Also, $\tilde{w}^n(t) = w^n(t(n)+t), t \in [0,T]$. Consider $\{\tilde{w}^n(\cdot)\}$ to be $C([0,T]; \R^{d_2} )$-valued random variables. 
\begin{lemma}\label{lem: slow_tightness}
	For $m^s(n) \geq \ell > k \geq n, n \geq 0$,
	\begin{equation*}
		E\left[\|w^n(t^s(\ell)) - w^n(t^s(k))\|^4\right]  =  O\left ( \left (\sum_{j=k}^{\ell}b(j)\right )^2 \right ) 
		=  O(|t^s(\ell) - t^s(k)|^2).
	\end{equation*}
\end{lemma}
The proof of this lemma can be found in Appendix \ref{sec: tightness_slow}. Tightness follows from Lemma \ref{lem: Lemma7.4}.

\subsection{Characterizing the limit} \label{sec: limit}
As proved in the previous section, the joint laws of $\{\tilde{u}^n(\cdot), \tilde{w}^n(\cdot), \tilde{y}^n(\cdot) \}$ are tight in $\mathcal{P}(C([0,T]; \mathbb{R}^d))^3$. We take a subsequence of $\{\tilde{u}^n(\cdot), \tilde{w}^n(\cdot), \tilde{y}^n(\cdot) \}$ (which we denote using the same notation) that converges in law to a limit $\{u^*(\cdot), w^*(\cdot), y^*(\cdot) \}$. Using Skorohod's theorem, without loss of generality we have that there is a sequence of random variables  (which we again denote using the same notation) on some probability space, with identical laws as those of $\{\tilde{u}^n(\cdot), \tilde{w}^n(\cdot), \tilde{y}^n(\cdot) \}$, that converges almost surely and the limiting random variable has the same law as that of $\{u^*(\cdot),w^*(\cdot), y^*(\cdot) \}$, i.e.,
\begin{align*}
	\left( \tilde{u}^n(\cdot), \tilde{w}^n(\cdot), \tilde{y}^n(\cdot) \right) \rightarrow \left(u^*(\cdot),w^*(\cdot), y^*(\cdot) \right) \ \text{a.s.}   
\end{align*}
As the trajectories of the o.d.e.\ $\dot{y}(t) = g(\lambda(y(t)), y(t))$ form a closed set in $C([0,T]; \mathbb{R}^d)$, $y^*(\cdot)$ also satisfies this o.d.e.\ and lies in an internally chain transitive invariant set of this o.d.e.\ by Theorem \ref{thm: as_limit}. In Sections \ref{sec: limit_slow_martingale} and \ref{sec: limit_fast_martingale}, we characterize $w^*(\cdot)$ and $u^*(\cdot)$ respectively.
\subsection{Characterizing the limit for the slow timescale} \label{sec: limit_slow_martingale}
For $j \geq n$, recall \eqref{eqn: alphan}:
\begin{align*}
	\alpha_{j+1}^n = \alpha_j^n + b(j) \left( \nabla \of(y^n(t^s(j))) \alpha_j^n + \delta_j +\Delta^s_j(x_j) +M^s_{j+1}+ e^{s,T}_j + p(x_j, y_j) \right).
\end{align*}
Using the Taylor expansion for $\og$, 
\begin{align*}
	p(x_j, y_j) &= \og(x_j,y_j) - \og(\lambda(y_j),y_j)  \\
	&= \nabla_x \og(\lambda(y_j), y_j) \left(x_j - \lambda(y_j) \right) + e^{s,T,p}_j,
\end{align*}
where $e^{s,T,p}_j = o(\| x_j - \lambda(y_j)\| )$. Thus, for $n \leq j \leq m^s(n)$, \eqref{eqn: alphan} is equivalent to:
\begin{align*}
	w_{j+1}^n &=\sqrt{\frac{a(j)}{a(j+1)}} w_j^n + b(j) \nabla \of(y^n(t^s(j)))  \sqrt{\frac{a(j)}{a(j+1)}} w_j^n + \frac{b(j)}{\sqrt{a(j+1)}} \Delta^s_j(x_j)\\&+ \sqrt{b(j)} \sqrt{\frac{a(j)}{a(j+1)}} \sqrt{\frac{b(j)}{a(j)}} M^s_{j+1} 	+ b(j) \sqrt{\frac{a(j)}{a(j+1)}} { \nabla_x \og(\lambda(y_j), y_j) u_j}+ o(b(j)) .
\end{align*}
Iterating the above equation for $n \leq j<m^s(n)$ gives:
\begin{multline}
	w_{j+1}^n = w_n^n + \sum_{k=n}^j \left( \sqrt{\frac{a(k)}{a(k+1)}} -1 \right) w_k^n + \sum_{k=n}^{j}  b(k) \nabla \of(y^n(t^s(k))) \sqrt{\frac{a(k)}{a(k+1)}} w_k^n \\
	+ \sum_{k=n}^{j}\frac{b(k)}{\sqrt{a(k+1)}} \Delta^s_k(x_k)+ \sum_{k=n}^{j} \sqrt{b(k)} \sqrt{\frac{a(k)}{a(k+1)}} \sqrt{\frac{b(k)}{a(k)}} M_{k+1}^s \\+ \sum_{k=n}^j b(k) \sqrt{\frac{a(k)}{a(k+1)}} {\nabla_x \og(\lambda(y_k), y_k) u_k} + o(1). \label{eqn: slow_limit_iterative_eqn}
\end{multline}
Define
\begin{align}
	c(t) &= \sqrt{\frac{a(i)}{a(i+1)}}, t \in [ t^s(i), t^s(i+1) ), i \geq 0, \\
	d_m^n &= \sum_{i=n}^m \left( \sqrt{\frac{a(i)}{a(i+1)}} -1 \right) w_i^n. 
\end{align}
We also have that for $n\leq m < m^s(n)$ (cf.\ \cite{Borkar2008}, Chapter 8)
\begin{align*}
	d_m^n &= \sum_{i=n}^m \left(  \frac{1}{2} \left (\frac{a(i)}{a(i+1)} -1  \right ) + O\left(  \left (\frac{a(i)}{a(i+1)} -1  \right )^2 \right) \right)w_i^n \numberthis\label{eqn: slow_drift} \\
	&= \sum_{i=n}^m \left(  \frac{b(i)}{2} \frac{a(i)}{b(i)} \left (\frac{1}{a(i+1)} -\frac{1}{a(i)}  \right ) + O\left(  \left (\frac{a(i)}{a(i+1)} -1  \right )^2 \right) \right)w_i^n.
\end{align*}
This can be shown to converge to 0 using \ref{assumption: bn_omega}. 
Thus,
\begin{multline*}
	w^n(t^s(j+1)) = w^n(t(n))  + \int_{t^s(n)}^{t^s(j+1)} \nabla \of(y^n(y)) w^n(y) c(y) dy \\
	+\sum_{k=n}^{j}\frac{b(k)}{\sqrt{a(k+1)}} \Delta^s_k(x_k)+ \sum_{k=n}^{j} \sqrt{b(k)} \sqrt{\frac{a(k)}{a(k+1)}} \sqrt{\frac{b(k)}{a(k)}} M_{k+1}^s \\+ \sum_{k=n}^j b(k) \sqrt{\frac{a(k)}{a(k+1)}} {\nabla_x \og(\lambda(y_k), y_k)u_k} + o(1).
\end{multline*}

Using the decomposition for $\Delta^s_k(x_k)$ from \eqref{eqn: delta_decomp_slow}, 
\begin{align*}
	&\sum_{k=n}^j \frac{b(k)}{\sqrt{a(k+1)}}\Delta^s_k(x_k) = \sum_{k=n}^j \sqrt{b(k)}\sqrt{\frac{b(k)}{a(k)}} \frac{\sqrt{a(k)}}{\sqrt{a(k+1)}}\zeta^s_k(x_k) \\&\hspace{8em}+\sum_{k=n}^j \frac{b(k)}{\sqrt{a(k+1)}} \left(\tau^s_k - \tau^s_{k+1}\right) + \sum_{k=n}^j b(k) \sqrt{\frac{a(k)}{a(k+1)}} \frac{e^{s,\Delta}_k(x_k)}{\sqrt{a(k)}}.  
\end{align*}
Using a summation by parts argument, we have that
\begin{align*}
	&\sum_{k=n}^j b(k)\sqrt{\frac{a(k)}{a(k+1)}}\frac{\left(\tau^s_{k+1} - \tau^s_{k}\right)}{\sqrt{a(k)}}  \\
	&= \frac{b(j)}{\sqrt{a(j)}}\sqrt{\frac{a(j)}{a(j+1)}} \tau^s_{j+1} - \frac{b(n)}{\sqrt{a(n)}}\sqrt{\frac{a(n)}{a(n+1)}} \tau^s_{n}  \\&\hspace{16em}- \sum_{k=n+1}^j  \tau^s_k \left( \frac{b(k)}{\sqrt{a(k+1)}} - \frac{b(k-1)}{\sqrt{a(k)}}\right) \\
	&= \epsilon_j \sqrt{a(j)} \sqrt{\frac{a(j)}{a(j+1)}} \tau^s_{j+1}  - \epsilon_n \sqrt{a(n)} \sqrt{\frac{a(n)}{a(n+1)}} \tau^s_{n} \\&\hspace{16em}- \sum_{k=n+1}^j \frac{b(k)}{\sqrt{a(k)}} \tau^s_k \left( \sqrt{\frac{a(k)}{a(k+1)}} - \frac{b(k-1)}{b(k)}\right).
\end{align*}
Thus,
\begin{align*}
	&E \left[ \left \| \sum_{k=n}^j b(k)\sqrt{\frac{a(k)}{a(k+1)}}\frac{\left(\tau^s_k - \tau^s_{k+1}\right)}{\sqrt{a(k)}} \right \|^2 \right] \\
	&\leq \newuck \left (a(j) \epsilon_j^2 + a(n) \epsilon_n^2 + E \left [ \left (\sum_{k=n+1}^j \frac{b(k)}{\sqrt{a(k)}}  \| \tau^s_k\|  \left| \sqrt{\frac{a(k)}{a(k+1)}} - \frac{b(k-1)}{b(k)}\right | \right ) ^2 \right ] \right )\\
	&\overset{(a)}{\leq} \newuck \left (\underset{k \geq n}{\sup} \ a(k)\epsilon_k^2 + T \sum_{k=n+1}^j b(k) \left (\frac{1}{\sqrt{a(k)}} \right )^2  E \left [\| \tau^s_j\|^2 \right ] \left | \sqrt{\frac{a(k)}{a(k+1)}} - \frac{b(k-1)}{b(k)}\right |^2   \right )\\
	&\overset{(b)}{\leq} \newuck \ \underset{k \geq n}{\sup} \ a(k)\epsilon_k^2 + \newuck \ \underset{k \geq n}{\sup} \ \frac{1}{a(k)} \left | \sqrt{\frac{a(k)}{a(k+1)}} - \frac{b(k-1)}{b(k)}\right |^2   
	\overset{n \uparrow \infty}{\rightarrow} 0, \numberthis \label{eqn: sum_of_parts_fast_limit}
\end{align*}
where $(a)$ follows from Jensen's inequality and $(b)$ follows from $\tau^s_j$ being bounded. The last step follows because $a(k) \overset{k \uparrow \infty}{\rightarrow} 0$ and
\begin{align*}
	&\frac{1}{a(k)}\left | \sqrt{\frac{a(k)}{a(k+1)}} - \frac{b(k-1)}{b(k)}\right |^2 \\
	&= \frac{1}{a(k)}\left | 1 + \frac{1}{2} \left ( \frac{a(k)}{a(k+1)} -1 \right ) + O \left( \left ( \frac{a(k)}{a(k+1)} -1 \right )^2 \right) - \frac{b(k-1)}{b(k)}\right |^2 \\
	&= \frac{1}{a(k)}\left | \frac{a(k)}{2} \left ( \frac{1}{a(k+1)} - \frac{1}{a(k)} \right ) + O \left( \left ( \frac{a(k)}{a(k+1)} -1 \right )^2 \right) +  b(k-1) \left ( \frac{1}{b(k-1)}- \frac{1}{b(k)} \right ) \right |^2 \\
	&\overset{(a)}{\leq} \frac{1}{a(k)}\left | \frac{a(k-1)}{2} \left ( \frac{1}{a(k+1)} - \frac{1}{a(k)} \right ) + O \left( a(k)^2 \left ( \frac{1}{a(k+1)} - \frac{1}{a(k)} \right )^2 \right) \right.\\&\hspace{20em}\left.- b(k-1) \left ( \frac{1}{b(k)} - \frac{1}{b(k-1)} \right ) \right |^2 \\ 
	&= \frac{a(k-1)^2}{a(k)}\left | \frac{1}{2} \left ( \frac{1}{a(k+1)} - \frac{1}{a(k)} \right ) + O \left( a(k) \left ( \frac{1}{a(k+1)} - \frac{1}{a(k)} \right )^2 \right) - \epsilon_{k-1} \left ( \frac{1}{b(k)} - \frac{1}{b(k-1)} \right ) \right |^2 \\  
	&\overset{(b)}{=} \frac{a(k-1)^2}{a(k)} \left | \frac{1}{2} \left ( \varphi + o(1) \right ) + O \left( a(k) \left ( \frac{1}{a(k+1)} - \frac{1}{a(k)} \right )^2 \right) -  \epsilon_{k-1}\left ( \vartheta + o(1)\right ) \right |^2 \\   
	&\overset{(c)}{\leq} a(k-1) \left | \frac{1}{2} \left ( \varphi + o(1) \right ) + O \left( a(k) \left ( \frac{1}{a(k+1)} - \frac{1}{a(k)} \right )^2 \right) -   \epsilon_{k-1}\left ( \vartheta + o(1)\right ) \right |^2 \overset{k \uparrow \infty}{\rightarrow} 0, 
\end{align*}
where $(a)$ follows from \ref{assumption: an_decreasing}, and $(b)$ and $(c)$ follow from \ref{assumption: alpha_beta}. 

Next, using the Lipschitz property of $V^s(\cdot,\cdot,\cdot)$, 
\begin{align*}
	\left \| e^{s,\Delta}_k(x_k) \right \| &\leq \newuck \| y_{k+1} - y_j \| \\
	&\leq b(k) \left \| g(x_k, y_k, Y_k) + M^s_{k+1} \right \|. 
\end{align*}
Using the above and Jensen's inequality: 
\begin{align*}
	&E \left[ \left \| \sum_{k=n}^j b(k)\sqrt{\frac{a(k)}{a(k+1)}}\frac{e^{s,\Delta}_k(x_k)}{\sqrt{a(k)}} \right \|^2 \right] \\
	&\leq T \sum_{k=n}^j b(k) \frac{a(k)}{a(k+1)} E \left [ \left \| \frac{e^{s,\Delta}_k(x_k)}{\sqrt{a(k)}}  \right \|^2 \right ] \\
	&\leq \newuck \sum_{k=n}^j b(k) \epsilon_k b(k) \leq \ \newuck \underset{k \geq n}{\sup} \ \epsilon_k b(k) \overset{n \uparrow \infty}{\rightarrow} 0. 
\end{align*}
Finally, note that $\xi^s_k := \zeta^s_k(x_k)+M^s_{k+1}$ is a martingale difference sequence as both terms on the right are, with respect to a common same filtration. We define, for $k \geq0$,
\begin{align*}
	Q_s(x_k, y_k) &= E \left[\xi^s_k {\xi^s_k}^{\intercal} | x_i, y_i, Y_i, i \leq k \right].
\end{align*}

It follows from \ref{assumption: martingale_markov_cond_independence}, \ref{assumption: martingale} and $V^s(\cdot, \cdot, \cdot)$ being Lipschitz that $Q_s(\cdot, \cdot)$ is Lipschitz. Fix $t, s \in [0,T]$ such that $t>s$ and let $\tilde{g} \in C_b(C([0,s]; \mathbb{R}^d)^3)$. As $\{\xi_k^s\}$ is a martingale difference sequence, 
\begin{multline*}
	\Biggl \| E \Biggl[ \left( \tilde{w}^n(t) - \tilde{w}^n(s) \right.\\ \left.- \int_{s}^{t}\left(c(t(n)+y) \left( \nabla \of(\tilde{y}^n(y))\tilde{w}^n(y) +\nabla_x \og(\lambda(\tilde{y}^n(y)),\tilde{y}^n(y)) \tilde{u}^n(y)\right)dy  \right)  \right)   \\ \times \tilde{g}\left(\tilde{y}^n([0,s]), \tilde{w}^n([0,s]), \tilde{u}^n([0,s])\right ) \Biggr] \Biggr \| = o(1).
\end{multline*}
Taking the limit $ n \rightarrow \infty$, we get 
\begin{multline*}
	\Biggl \| E \Biggl[ \left( w^*(t) - w^*(s) - \int_{s}^{t} \left( \nabla \of(y^*(y))w^*(y) + \nabla_x \og(\lambda(y^*(y)),y^*(y)) u^*(y) \right) dy \right)   \\ \times g\left(w^*([0,s]), y^*([0,s]), u^*([0,s])\right ) \Biggr] \Biggr \| = 0.
\end{multline*}
It follows from a standard monotone class argument that 
\begin{align*}
	w^*(t)  - \int_{0}^{t} \left( \nabla \of(y^*(y))w^*(y) +\nabla_x \og(\lambda(y^*(y)),y^*(y)) u^*(y)  \right) dy
\end{align*}
is a martingale with respect to the filtration  given by the completion of $\cap_{r \geq t} \sigma(w^*(y), y^*(y), u^*(y), y \leq r)$ for $r \geq 0$. For $t \in [0,T]$, define $\Sigma_s^n(t)$ by 
\begin{align*}
	\Sigma^n_s(t^s(j) - t^s(n)) = \sum_{k=n}^j b(k) c(t^s(k))^2 \epsilon_k Q_s(\lambda(y^n(t^s(k))),y^n(t^s(k))),
\end{align*}
for $n \leq j \leq m^s(n)$ and linear interpolation on each interval $[t^s(j)-t^s(n),t^s(j+1)-t^s(n)]$. Then
\begin{align*}
	\sum_{k=n}^{j} b(k) \epsilon_k c(t^s(k))^2 \xi_{k}\xi_{k}^{\intercal} - \Sigma^n_s(t^s(j) - t^s(n)),
\end{align*}
is a martingale. Thus, for $t \in [0,T]$ and 
\begin{align*}
	q_s^n(t) := \tilde{w}^n(t) - \int_{0}^{t} c(t(n)+y) \left( \nabla \of(\tilde{y}^n(y))\tilde{w}^n(y) +\nabla_x \og(\lambda(\tilde{y}^n(y)),\tilde{y}^n(y)) \tilde{u}^n(y)  \right) dy,
\end{align*}
we have that 
\begin{align*}
	&\left \| E \left[ q_s^n(t)q_s^n(t)^{\intercal} - q^n_s(s)q^n_s(s)^{\intercal} - \left( \Sigma_s^n(t) - \Sigma_s^n(s) \right) \right.\right. \\&\hspace{10em}\left.\left. \times g(\tilde{w}([0,s]), \tilde{y}([0,s]), \tilde{u}^n([0,s]) ) \right] \right \| = o(1). 
\end{align*}
Taking the limit $n \rightarrow \infty$, we have that 
\begin{eqnarray*}
	&&\left( 	w^*(t)  - \int_{0}^{t} \left(  \nabla\of(y^*(y))w^*(y) + \nabla_x \og(\lambda(y^*(y)),y^*(y)) u^*(y)  \right) dy  \right) \\ 
	&\times&  \left( 	w^*(t)  - \int_{0}^{t} \left(  \nabla \of(y^*(y))w^*(y) + \nabla_x \og(\lambda(y^*(y)),y^*(y)) u^*(y) \right) dy \right)^{\intercal}
\end{eqnarray*}
is a martingale, i.e., the quadratic variation process of the martingale 
\begin{align*}
	w^*(t)  - \int_{0}^{t} \left( \nabla \of(y^*(y))w^*(y) + \nabla_x \og(\lambda(y^*(y)),y^*(y)) u^*(y)  \right) dy,
\end{align*}
is zero, which, in view of the continuity of paths,  implies that this martingale is identically zero. We thus have that
\begin{align*}
	w^*(t)  = \int_{0}^{t} \left( \nabla  \of(y^*(y))w^*(y) + \nabla_x \og(\lambda(y^*(y)),y^*(y)) u^*(y)  \right) dy.
\end{align*}
Note that this is an ordinary differential equation, i.e., $w^*(\cdot)$ is a deterministic trajectory. 

\subsection{Characterizing the limit for the fast timescale} \label{sec: limit_fast_martingale}

In this section, we reintroduce the quantity $\varphi$ from Assumption \ref{assumption: alpha_beta},  although we have asumed this quantity to be zero thus far. This is with an eye on the next section, Section \ref{sec: alpha}, which  specifically analyzes the case where it is not. The use of $\varphi$ here helps facilitate the claims there which depend on the calculations of this section. 

Recall \eqref{eqn: betan_eqn} from Section \ref{sec: prelims_iterative_eqns}. For $n \geq 0$:
\begin{align*}
	\beta_{n+1} &=\beta_n + a(n) \left( \nabla_x \oh(x_n, y_n ) \beta_n + \epsilon_n\nabla \lambda(y_n) \og(x_n, y_n) + \epsilon_n \nabla \lambda(y_n)M^s_{n+1} \right.\\&\hspace{10em}\left.+ \epsilon_n \nabla \lambda(y_n)\Delta^s_n(x_n) + \frac{e_{n+1}}{a(n)} - \Delta^f_n(y_n) - M^f_{n+1}+ e^{f,T}_n \right). 
\end{align*}
This is equivalent to, for $n \leq j \leq m^f(n)$,
\begin{multline*}
	u_{j+1} =\sqrt{\frac{a(j)}{a(j+1)}} u_j + a(j) \nabla_x \oh(x_j, y_j) \sqrt{\frac{a(j)}{a(j+1)}} u_j - \sqrt{a(j)} \sqrt{\frac{a(j)}{a(j+1)}}  M_{k+1}^f \\- \sqrt{a(j)} \sqrt{\frac{a(j)}{a(j+1)}} \Delta^f_j(y_j)
	+ a(j) \sqrt{\frac{a(j)}{a(j+1)}} \frac{b(j)}{a(j)^{3/2}}\nabla \lambda(y_j) M_{j+1}^s\\+ a(j) \sqrt{\frac{a(j)}{a(j+1)}} \frac{b(j)}{a(j)^{3/2}}\nabla \lambda(y_j)\Delta^s_j(x_j)+a(j)\sqrt{\frac{a(j)}{a(j+1)}} \frac{e_{j+1}}{a(j)^{3/2}}+ o(a(j)).
\end{multline*}
Next, we iterate the above equation. Using the Lipschitzness of $\nabla_x \oh(\cdot, \cdot)$ and  the result that $x_k - \lambda(y_k) \overset{k \uparrow \infty}{\rightarrow} 0$, we have that
\begin{multline*}
	u_{j+1} = u_n + \sum_{k=n}^j \left( \sqrt{\frac{a(k)}{a(k+1)}} -1 \right) u_k + \sum_{k=n}^{j}  a(k) \nabla_x \oh(\lambda(y_k), y_k) \sqrt{\frac{a(k)}{a(k+1)}} u_k \\- \sum_{k=n}^j \sqrt{a(k)} \sqrt{\frac{a(k)}{a(k+1)}}  M_{k+1}^f 
	+ \sum_{k=n}^j a(k) \sqrt{\frac{a(k)}{a(k+1)}} \frac{b(k)}{a(k)^{3/2}}\nabla \lambda(y_k) M_{k+1}^s \\ - \sum_{k=n}^j \sqrt{a(k)} \sqrt{\frac{a(k)}{a(k+1)}} \Delta^f_k(y_k) + \sum_{k=n}^j a(k) \sqrt{\frac{a(k)}{a(k+1)}} \frac{b(k)}{a(k)^{3/2}}\nabla \lambda(y_k)\Delta^s_k(x_k)
	\\+ \sum_{k=n}^j a(k) \sqrt{\frac{a(k)}{a(k+1)}} \frac{e_{k+1}}{a(k)^{3/2}} + o(1).
\end{multline*}
Define
\begin{align*}
	c(t) &= \sqrt{\frac{a(i)}{a(i+1)}}, t \in [ t^f(i), t^f(i+1) ), i \geq 0, \\
	d_m^n &= \sum_{i=n}^m \left( \sqrt{\frac{a(i)}{a(i+1)}} -1 \right) u_i.
\end{align*}
For $m \geq n$,
\begin{align*}
	d_m^n &= \sum_{i=n}^{m} a(i) \left( \frac{\varphi}{2} + o(1)  \right)u_i.
\end{align*}
Thus
\begin{multline*}
	u(t^f(j+1)) = u(t^f(n)) + d_{j}^n - d_n^n + \int_{t^f(n)}^{t^f(j+1)} \nabla_x \oh(\lambda(\tilde{y}^n(y)), \tilde{y}^n(y ) )  \tilde{u}^n(y) c(y) dy \\
	- \sum_{k=n}^j \sqrt{a(k)} \sqrt{\frac{a(k)}{a(k+1)}}  M_{k+1}^f 
	+ \sum_{k=n}^j a(k) \sqrt{\frac{a(k)}{a(k+1)}} \frac{b(k)}{a(k)^{3/2}}\nabla \lambda(y_k) M_{k+1}^s \\ - \sum_{k=n}^j \sqrt{a(k)} \sqrt{\frac{a(k)}{a(k+1)}} \Delta^f_k(y_k) + \sum_{k=n}^j a(k) \sqrt{\frac{a(k)}{a(k+1)}} \frac{b(k)}{a(k)^{3/2}}\nabla \lambda(y_k)\Delta^s_k(x_k)\\
	+ \sum_{k=n}^j a(k) \sqrt{\frac{a(k)}{a(k+1)}} \frac{e_{k+1}}{a(k)^{3/2}} + o(1).
\end{multline*}
Now,
\begin{align*}
	&E \left[ \left \| \sum_{k=n}^j a(k) \sqrt{\frac{a(k)}{a(k+1)}} \frac{b(k)}{a(k)^{3/2}}\nabla \lambda(y_k) M_{k+1}^s \right \|^2 \right] \\
	&\leq 	E \left[  \left ( \sum_{k=n}^j a(k) \sqrt{\frac{a(k)}{a(k+1)}} \frac{b(k)}{a(k)^{3/2}} \left \| \nabla \lambda(y_k) M_{k+1}^s \right \| \right )^2 \right] \\
	&\leq 	\left(\sum_{k=n}^j a(k) \right) E \left[  \sum_{k=n}^j a(k) \frac{a(k)}{a(k+1)} \left ( \frac{b(k)}{a(k)^{3/2}} \right )^2 \left \| \nabla \lambda(y_k) M_{k+1}^s \right \|^2  \right] \\
	&\leq \newuck \ \underset{k \geq n}{\sup} \left ( \frac{b(k)}{a(k)^{3/2}} \right )^2 \overset{n \uparrow \infty}{\rightarrow} 0,
\end{align*}
using Jensen's inequality and the fact that the second moment of $M^s_{k+1}$ is bounded. Similarly,
\begin{align*}
	E \left[ \left \| \sum_{k=n}^j a(k) \sqrt{\frac{a(k)}{a(k+1)}} \frac{e_{k+1}}{a(k)^{3/2}}  \right \|^2 \right] &\leq 	\left(\sum_{k=n}^j a(k) \right) E \left[ \sum_{k=n}^j a(k) \frac{a(k)}{a(k+1)}  \left \| \frac{e_{k+1}}{a(k)^{3/2}}  \right \|^2 \right] \\
	&\leq \newuck \ \underset{k \geq n}{\sup} \ \epsilon_k^4 a(k)^2   \overset{n \uparrow \infty}{\rightarrow} 0,
\end{align*}
using \eqref{eqn: en_bound}, the boundedness of the second moment of $M^s_{k+1}$ and $b(k) < a(k)^{3/2}$.

Using the boundedness of $\Delta^s_j(x_j)$ and Jensen's inequality,
$$
\begin{aligned}
	E \left[ \left \| \sum_{k=n}^j \sqrt{a(k)}\sqrt{\frac{a(k)}{a(k+1)}}\epsilon_k \nabla \lambda(y_k)\Delta^s_k(x_k)  \right \|^2 \right] &\leq \newuck\label{const: fast_delta_s_ub} \sum_{k=n}^j a(k) \frac{\epsilon_k^2}{a(k)} \\
	&= \olduck{const: fast_delta_s_ub} \sum_{k=n}^j a(k) \left (\frac{b(k)}{a(k)^{3/2}} \right )^2
	\overset{n \uparrow \infty}{\rightarrow} 0. 
\end{aligned}
$$
Using the decomposition for $\Delta^f_j(y_j)$ from \eqref{eqn: delta_decomp_fast}, 
\begin{multline*}
	\sum_{k=n}^j \frac{a(k)}{\sqrt{a(k+1)}}\Delta^f_k(y_k) = \sum_{k=n}^j \sqrt{a(k)} \frac{\sqrt{a(k)}}{\sqrt{a(k+1)}}\zeta^f_k(y_k) \\+\sum_{k=n}^j \frac{a(k)}{\sqrt{a(k+1)}} \left(\tau^f_k - \tau^f_{k+1}\right) + \sum_{k=n}^j a(k) \sqrt{\frac{a(k)}{a(k+1)}} \frac{e^{f,\Delta}_k(y_k)}{\sqrt{a(k)}}. 
\end{multline*}
Using a similar summation by parts argument as in \eqref{eqn: sum_of_parts_fast_limit}: 
\begin{align*}
	&\sum_{k=n}^j \frac{a(k)}{\sqrt{a(k+1)}}\left(\tau^f_{k+1} - \tau^f_{k}\right)  \\
	&= \frac{a(j)}{\sqrt{a(j+1)}} \tau^f_{j+1} - \frac{a(n)}{\sqrt{a(n+1)}} \tau^f_{n} - \sum_{k=n+1}^j  \tau^f_k \left( \frac{a(k)}{\sqrt{a(k+1)}} - \frac{a(k-1)}{\sqrt{a(k)}}\right) \\
	&=  \sqrt{a(j)} \sqrt{\frac{a(j)}{a(j+1)}} \tau^f_{j+1}  - \sqrt{a(n)} \sqrt{\frac{a(n)}{a(n+1)}} \tau^f_{n} \\&\hspace{16em}- \sum_{k=n+1}^j \sqrt{a(k)} \tau^f_k \left( \sqrt{\frac{a(k)}{a(k+1)}} - \frac{a(k-1)}{a(k)}\right).
\end{align*}
Thus,
\begin{align*}
	&E \left[ \left \| \sum_{k=n}^j \frac{a(k)}{\sqrt{a(k+1)}}\left(\tau^f_{k+1} - \tau^f_{k}\right)  \right \|^2 \right] \\
	&\leq \newuck \left (a(j)  + a(n)+ E \left [ \left (\sum_{k=n+1}^ja(k) \frac{1}{\sqrt{a(k)}} \| \tau^f_k\|  \left| \sqrt{\frac{a(k)}{a(k+1)}} - \frac{a(k-1)}{a(k)}\right | \right ) ^2 \right ] \right )\\
	&\overset{(a)}{\leq} \newuck \left (\underset{k \geq n}{\sup} \ a(k) + T \sum_{k=n+1}^j a(k) \left ( \frac{1}{\sqrt{a(k)}} \right )^2   E \left [\| \tau^f_j\|^2 \right ] \left | \sqrt{\frac{a(k)}{a(k+1)}} - \frac{a(k-1)}{a(k)}\right |^2   \right )\\
	&\overset{(b)}{\leq} \newuck \ \underset{k \geq n}{\sup} \ a(k)  + \newuck \ \underset{k \geq n}{\sup} \ \frac{1}{a(k)}\left | \sqrt{\frac{a(k)}{a(k+1)}} - \frac{a(k-1)}{a(k)}\right |^2
	\overset{n \uparrow \infty}{\rightarrow} 0, 
\end{align*}
where $(a)$ follows from Jensen's inequality and $(b)$ follows from $\tau^f_j$ being bounded and the fact that $\sum_{n+1}^{j} a(k) \leq T$. The last step follows because $a(k) \overset{k \uparrow \infty}{\rightarrow} 0$ and
\begin{align*}
	&\frac{1}{a(k)}\left | \sqrt{\frac{a(k)}{a(k+1)}} - \frac{a(k-1)}{a(k)}\right |^2 \\
	&= \frac{1}{a(k)}\left | 1 + \frac{1}{2} \left ( \frac{a(k)}{a(k+1)} -1 \right ) + O \left( \left ( \frac{a(k)}{a(k+1)} -1 \right )^2 \right) - \frac{a(k-1)}{a(k)}\right |^2 \\
	&= \frac{1}{a(k)}\left | \frac{a(k)}{2} \left ( \frac{1}{a(k+1)} - \frac{1}{a(k)} \right ) + O \left( \left ( \frac{a(k)}{a(k+1)} -1 \right )^2 \right) +  a(k-1) \left ( \frac{1}{a(k-1)}- \frac{1}{a(k)} \right ) \right |^2 \\
	&\overset{(a)}{\leq} \frac{1}{a(k)}\left | \frac{a(k-1)}{2} \left ( \frac{1}{a(k+1)} - \frac{1}{a(k)} \right ) + O \left( a(k)^2 \left ( \frac{1}{a(k+1)} - \frac{1}{a(k)} \right )^2 \right) \right.\\&\hspace{20em}\left.- a(k-1) \left ( \frac{1}{a(k)} - \frac{1}{a(k-1)} \right ) \right |^2 \\ 
	&= \frac{a(k-1)^2}{a(k)}\left | \frac{1}{2} \left ( \frac{1}{a(k+1)} - \frac{1}{a(k)} \right ) + O \left( a(k) \left ( \frac{1}{a(k+1)} - \frac{1}{a(k)} \right )^2 \right) -  \left ( \frac{1}{a(k)} - \frac{1}{a(k-1)} \right ) \right |^2 \\  
	&\overset{(b)}{=} \frac{a(k-1)^2}{a(k)} \left | \frac{1}{2} \left ( \varphi + o(1) \right ) + O \left( a(k) \left ( \frac{1}{a(k+1)} - \frac{1}{a(k)} \right )^2 \right) -  \left ( \varphi + o(1)\right ) \right |^2 \\   
	&\overset{(c)}{\leq} a(k-1) \left | \frac{1}{2} \left ( \varphi + o(1) \right ) + O \left( a(k) \left ( \frac{1}{a(k+1)} - \frac{1}{a(k)} \right )^2 \right) -  \left ( \varphi + o(1)\right ) \right |^2 \overset{k \uparrow \infty}{\rightarrow} 0, \\   
\end{align*}
where $(a)$ follows from \ref{assumption: an_decreasing}, and $(b)$ and $(c)$ follow from \ref{assumption: alpha_beta}. 

Next, using the Lipschitz property of $V^f(\cdot,\cdot,\cdot)$, 
\begin{align*}
	\left \| e^{f,\Delta}_k(y_k) \right \| &\leq \newuck \| x_{k+1} - x_k \| \\
	&\leq a(k) \left \| h(x_k, y_k, Y_k) + M^f_{k+1} \right \|. 
\end{align*}
Using the above and Jensen's inequality: 
\begin{align*}
	&E \left[ \left \| \sum_{k=n}^j a(k)\sqrt{\frac{a(k)}{a(k+1)}}\frac{e^{f,\Delta}_k(y_k)}{\sqrt{a(k)}} \right \|^2 \right] \\
	&\leq T \sum_{k=n}^j a(k) \frac{a(k)}{a(k+1)} E \left [ \left \| \frac{e^{f,\Delta}_k(y_k)}{\sqrt{a(k)}}  \right \|^2 \right ] \\
	&\leq \newuck \sum_{k=n}^j a(k) a(k) \leq \ \newuck \ \underset{k \geq n}{\sup} \ a(k) \overset{n \uparrow \infty}{\rightarrow} 0. 
\end{align*}
Finally, note that $\xi^f_k := \zeta^f_k(x_k)+M^f_{k+1}$ is a martingale difference sequence as both have the same filtration. We define, for $k \geq0$,
\begin{align*}
	Q_f(x_k, y_k) &= E \left[\xi^f_k {\xi^f_k}^{\intercal} | x_i, y_i, Y_i, i \leq k \right].
\end{align*}

It follows from \ref{assumption: martingale_markov_cond_independence}, \ref{assumption: martingale} and $V^f(\cdot, \cdot, \cdot)$ being Lipschitz that $Q_f(\cdot, \cdot)$ is Lipschitz. It can be shown in a manner similar to that in Section \ref{sec: limit_fast_martingale} that for $t \in [0,T]$ and for some $y(\cdot) \in C\left([0,T]; \R^{d_1}\right)$, 
\begin{align*}
	u^*(t) - \int_{0}^{t} \left ( \nabla_x h(\lambda(y(s)),y(s)) + \frac{\varphi}{2} \right )u^*(s) ds,
\end{align*}
is a martingale with respect to the filtration $\cap_{r \geq t} \sigma(w^*(s), y(s), u^*(s), s \leq r)$ for $r \geq 0$. For $t \in [0,T]$, define $\Sigma_f^n(t)$ by 
\begin{align*}
	\Sigma^n_f(t^f(j) - t^f(n)) = \sum_{k=n}^j a(k) c(t^f(k))^2 Q_f(\lambda(y^n(t^f(k))),y^n(t^f(k))),
\end{align*}
for $n \leq j \leq m^f(n)$ and linear interpolation on each interval $[t^f(j)-t^f(n),t^f(j+1]-t^f(n))$. Then
\begin{align*}
	\sum_{k=n}^{j} a(k) c(t^f(k))^2 \xi^f_{k}{\xi^f_{k}}^{\intercal} - \Sigma^n_f(t^f(j) - t^f(n)),
\end{align*}
is a martingale. Thus, for $t \in [0,T]$ and 
\begin{align*}
	q_f^n(t^f(j)) &:= \tilde{u}^n(t^f(j)) \\
	&- \int_{0}^{t^f(j)} c(t^f(n)+y) \left( \nabla_x \oh(\overline{x}(t^f(n)+y), \overline{y}(t^f(n)+y))\tilde{u}^n(y) +\frac{\varphi}{2}I  \right) dy,
\end{align*}
with linear interpolation on each interval $[t^f(j)-t^f(n),t^f(j+1)-t^f(n)]$, we have that 
\begin{align*}
	&\left \| E \left[ q_f^n(t)q_f^n(t)^{\intercal} - q^n_f(s)q^n_f(s)^{\intercal} - \left( \Sigma^n_f(t) - \Sigma^n_f(s) \right) \right.\right. \\&\hspace{5em} \left.\left. \times g(\tilde{w}([0,s]), \tilde{y}([0,s]), \tilde{u}^n([0,s]) ) \right] \right \| = o(1). 
\end{align*}
Taking the limit $n \rightarrow \infty$, we have that the limit points in law of $\left\{\left(\tilde{u}^n(\cdot), \tilde{w}^n(\cdot), \tilde{y}^n(\cdot)\right)\right\}$ satisfy
\begin{multline*}
	\left( 	u^*(t)  -  \int_{0}^{t} \left ( \nabla_x \oh(\lambda(y(s)),y(s)) + \frac{\varphi}{2} \right )u^*(s) ds \right) \\ \times \left( 	u^*(t)  -  \int_{0}^{t} \left ( \nabla_x \oh(\lambda(y(s)),y(s)) + \frac{\varphi}{2} \right )u^*(s) ds \right)^{\intercal} - \int_{0}^{t} Q_f(\lambda(y(s)), y(s))ds
\end{multline*}
is a martingale.
Using Theorem 4.2 of \cite{karatzas_shreve_martingale_representation}, we have that on a possibly augmented probability space, there exists a $d$-dimensional Brownian motion $B(t), t \geq 0,$ such that  
\begin{align*}
	u^*(t)  = \int_{0}^{t} \left ( \nabla_x \oh(\lambda(y(s)),y(s)) + \frac{\varphi}{2} \right )u^*(s) ds + \int_{0}^{t} G(\lambda(y(s))) dB(s),
\end{align*}
where $G(x)$ is a positive semidefinite, Lipschitz, square-root of the matrix $Q_f(x)$. 

This concludes the proof of Theorem \ref{thm: main_thm}.

\section{The case when $\varphi$ need not be zero} \label{sec: alpha}
If $\varphi \neq 0$, then \ref{assumption: bn_omega} does not hold. The primary reason why the preceding proof does not work for this case is due to \eqref{eqn: slow_drift}: If $\varphi \neq 0$, then the first term inside the summation in \eqref{eqn: slow_drift} grows unboundedly. To circumvent this, one can rewrite the iterative equations for the fluctuations in the slow timescale on the $\{a(n)\}$ timescale. To be precise, the iteration on the $\{b(n)\}$ timescale \eqref{eqn: iteration2} can rewritten on the $\{a(n)\}$ timescale as:
\begin{align}
	y_{n+1} = y_n + a(n) \left[ \epsilon_n \of(y_n) +\epsilon_n p(x_n, y_n)+ \epsilon_n \Delta^s_n(x_n)+ \epsilon_n M^s_{n+1} \right]. \label{eqn: equation2_alpha}
\end{align}
Define for $ j \geq n$,
\begin{align*}
	w_j^n = \frac{y_j - y^n(t^f(j))}{\sqrt{a(j)}}. 
\end{align*}
Redefine the piecewise linear interpolation $w^n(t), t \in \left[t^f(n), t^f(n)+T\right]$ for $\{y_n\}$ on the $\{a(n)\}$ timescale such that for $n \leq j \leq m^f(n)$, $w^n(t^f(j)) = w_j^n$ and linearly interpolated otherwise. 

The resultant limit points in law in this case is summarized in the following theorem:
\begin{theorem} \label{thm: main_thm_alpha}
	The limit points in law $(u^*(\cdot),w^*(\cdot), y^*(\cdot))$ of $\{\tilde{u}^n(\cdot), \tilde{w}^n(\cdot), \tilde{y}^n(\cdot) \}$ are such that $u^*(\cdot)$ satisfies
	\begin{align}
		u^*(t)  = \int_{0}^{t} \left ( \nabla_x h(\lambda(y^*(s)),y^*(s)) + \frac{\varphi}{2} I \right )u^*(s) ds + \int_{0}^{t} G(\lambda(y^*(s)), y^*(s)) dB(s), \label{eqn: fast_sde_alpha}
	\end{align}
	where $G(\cdot, \cdot)$ is a positive semidefinite, Lipschitz matrix, and $w^*(\cdot)$ satisfies
	\begin{align}
		w^*(t)  = \int_{0}^{t} \frac{\varphi}{2}w^*(s) ds. \label{eqn: slow_ode_alpha}
	\end{align}
	and $y^*(\cdot) \equiv y', y' \in \R^d$.
\end{theorem}
The proof of Theorem \ref{thm: main_thm_alpha} can be found in Appendix \ref{sec: main_thm_alpha}. 

\section{A Central Limit Theorem} \label{sec: clt}
Consider the case where the o.d.e.\ \eqref{eqn: second_timescale_ode} has a globally asymptotically stable equilibrium $y^*$ that is also exponentially stable, i.e., the solution $y(\cdot)$ of the o.d.e. \eqref{eqn: second_timescale_ode}  satisfies $\left \| y(t) - y^*\right \| \leq ce^{-dt} \left \| y(0) - y^*\right \|$. We will also be working under the assumption that $\varphi=0$. In this section, we briefly sketch a central limit theorem (CLT) from the FCLT proved in Theorem \ref{thm: main_thm_alpha} along the lines of \cite{borkar_meyn_markov}, for the following quantities:
\begin{align}
	\eta_n &:= \frac{\lambda(y^*) - x_n }{\sqrt{a(n)}}, \label{cltone}\\
	\upsilon_n &:= \frac{y_n - y^*}{\sqrt{a(n)}}. \label{clttwo}
\end{align} 
From Theorem \ref{thm: main_thm_alpha} with $\varphi = 0$, we have that for any continuous and bounded functions $g : \R^{d_1} \rightarrow \R$ and $g' : \R^{d_2} \rightarrow \R$, 
\begin{align}
	\underset{n \uparrow \infty}{\lim} \ E \left[ g(u^n_{m^f(n)}) \right] &= E \left[ g ( u^*_T) \right] \label{eqn: CLT_slow}\\
	\underset{n \uparrow \infty}{\lim} \ E \left[ g'(w^n_{m^f(n)}) \right] &= E \left[ g' ( w^*_T) \right], \label{eqn: CLT_fast}
\end{align}
where $u^*_T$ and $w^*_T$ are Gaussian random variables, viz.\  the solutions of \eqref{eqn: fast_sde_alpha} and \eqref{eqn: slow_ode_alpha} respectively (with appropriate initial conditions). Note that for $\varphi=0$, $w^*_T = 0$. To obtain the CLT using the limits in law of $u^n_{m^f(n)}$ and $w^n_{m^f(n)}$, we rewrite their equations in the following manner:
\begin{align*}
	w^n_{m^f(n)} &= \frac{ y_{m^f(n)} - y^n(t^f(m^f(n)))}{\sqrt{a(m^f(n))}} = \frac{y^* - y_{m^f(n)}}{\sqrt{a(m^f(n))}} + \frac{y^n(t^f(m^f(n))) - y^*}{\sqrt{a(m^f(n))}}, \\
	u^n_{m^f(n)} &= \frac{ \lambda(y_{m^f(n)}) - x_{m^f(n)}  }{\sqrt{a(m^f(n))}} =\frac{ \lambda(y^*) - x_{m^f(n)}  }{\sqrt{a(m^f(n))}} + \frac{ \lambda(y^n(t^f(m^f(n))))-\lambda(y^*)  }{\sqrt{a(m^f(n))}}  \\&\hspace{13em}+ \frac{ \lambda(y_{m^f(n)})- \lambda(y^n(t^f(m^f(n))))  }{\sqrt{a(m^f(n))}}.
\end{align*}
First consider the equation for the slow timescale. From our assumption that \eqref{eqn: second_timescale_ode} is globally exponentially stable, it follows that
\begin{align*}
	\left \| \frac{y^n(t^f(m^f(n))) - y^*}{\sqrt{a(m^f(n))}} \right \| &\leq c e^{-dT} \frac{\left \| y_n - y^*\right \|}{\sqrt{a(m^f(n))}} \\
	&= c e^{-dT} \| \upsilon_n \| \sqrt{\frac{ a(n)}{a(m^f(n))}}.
\end{align*}
Using \ref{assumption: fast_ode_convergence}, we have that
\begin{align*}
	\left \| w^n_{m^f(n)} - \upsilon_{m^f(n)}\right \| &\leq c e^{-dT} \| \upsilon_n\|  \sqrt{\frac{ a(n)}{a(m^f(n))}} \\
	&\leq \zeta(T) \| \upsilon_n \|, 
\end{align*}
where $\zeta(T)$ is a decreasing function of $T$ decreasing to zero, independent of $n$. Let $b^s_{n} := \underset{ n \leq k \leq m^f(n)}{\max} E \left [ \| \upsilon_k \|^4 \right ] $. We will first show that $\{E \left[ \| \upsilon_n \|^4 \right], n \geq 0\}$ is bounded. It follows from the above equation that
\begin{align*}
	\|  \upsilon_n \| \leq  \| w^{m^{f-}(n)}_{n}  \| +\zeta(T)^4  \| \upsilon_{m^{f-}(n)} \|.
\end{align*}
Using \eqref{eqn: slow_fourth_moment_bdd},
\begin{align*}
	b^s_n &= \underset{ n \leq k \leq m^f(n)}{\max}	E \left [ \left \| \upsilon_{k} \right \|^4 \right ] \leq 	\underset{ n \leq k \leq m^f(n)}{\max} 8 E \left [ \left \| w^{m^{f-}(k)}_{k} \right \|^4 \right ] + 8 \zeta(T)  \underset{ n \leq k \leq m^f(n)}{\max} E \left [ \| \upsilon_{m^{f-}(k)} \|^4 \right ] \\
	&\leq \newuck +  8 \zeta(T) b^s_{m^{f-}(n)}. 
\end{align*}
Choosing $T$ large enough such that $\zeta(T)^4 < 1/8$, we have that $\{E \left[ \| \upsilon_n \|^4 \right]\}$ is bounded. It follows that 
\begin{align*}
	\underset{n \uparrow \infty}{\lim} 	E \left [ \left \| w^n_{m^f(n)} - \upsilon_{m^f(n)}\right \|^4 \right ] \leq \newuck \zeta(T)^4. 
\end{align*}
Taking $T \uparrow \infty$ yields the CLT for (\ref{cltone}) by \eqref{eqn: CLT_slow}. Next, we will derive a CLT for the fast timescale. Using Taylor's theorem, we have that
\begin{align*}
	u^n_{m^f(n)} &= \eta_{m^f(n)} + \frac{  \lambda(y^n(t^f(m^f(n)))) - \lambda(y^*)}{\sqrt{a(m^f(n))}}  + \nabla \lambda (y_{m^f(n)}) w^{n}_{m^f(n)}  + o(1),
\end{align*}
It is clear that $\nabla \lambda (y_{m^f(n)}) w^{n}_{m^f(n)}$ converges to 0 in probability as $\varphi=0$.  By the exponential stability of \eqref{eqn: second_timescale_ode}, it follows that
$$
\begin{aligned}
	\left \| 	u^n_{m^f(n)}  - \eta_{m^f(n)}  \right \| &\leq \frac{ \|  \lambda(y^*)- \lambda(y^n(t^f(m^f(n)))) \| }{\sqrt{a(m^f(n))}} + \left \|\nabla \lambda (y_{m^f(n)}) \right \|  \left \| w^{'n}_{m^f(n)} \right \| \\
	&\leq \newuck c e^{-dT} \| \upsilon_n \| \sqrt{\frac{ a(n)}{a(m^f(n))}} +  \left \|\nabla \lambda (y_{m^f(n)}) \right \| \left \| w^{'n}_{m^f(n)} \right \| .
\end{aligned}
$$
The CLT in \eqref{clttwo} follows from the result that $\{E \left[ \| \upsilon_n \|^4 \right]\}$ is bounded using the same arguments as above in view of \eqref{eqn: CLT_fast}.

For $\varphi > 0$, $w^*(T) \uparrow \infty$, which poses a problem.


%

\appendix
\section{Proof of Lemma \ref{lem: first_timescale_with_ode_fclt}} \label{sec: appendix_martingale_gamma}
	\begin{proof}[Proof of Lemma \ref{lem: first_timescale_with_ode_fclt}.]
	Let $q,r \in \mathbb{N}$ be such that $n < q < r \leq m^f(n)$. For $n \leq j \leq m^f(n)$,
	\begin{align*}
		x^n(t^f(j+1)) = x^n(t^f(j)) + a(j) \left(\oh(x^n(t^f(j)), y_n) - \delta^x_j \right),
	\end{align*}
	where $\delta^x_j = O(a(j))$ is the discretization error. 
	Subtracting the above equation from: 
	\begin{align*}
		x_{j+1} = x_j + a(j) \left(\oh(x_j, y_j) +\Delta^f_j(x_j)+ M^f_{j+1} \right),
	\end{align*}
	we have, for $\rho^n_k := x_k - x^n(t^f(k))$,  that
	\begin{align*}
		\rho^n_{j+1} &= \rho^n_j + a(j) \left( \oh(x_j, y_j) - \oh(x^n(t^f(j)), y_n) + M^f_{j+1} +\delta^x_j \right) \\
		&= \rho^n_j + a(j) \left( \nabla_x \oh(x^n(t^f(j)), y_n) \rho^n_j + \nabla_y \oh(x^n(t^f(j)), y_n) \left(y_j - y_n\right) \right. \\&\hspace{15em} \left.+\kappa^x_j + \kappa^y_j + M^f_{j+1} +\Delta^f_j(x_j)+\delta^x_j \right),
	\end{align*}
	where $\kappa^x_j = o( \| \rho^n_j  \|)$, $\kappa^y_j = o( \left  \| y_j - y_n \right \|)$ and $\delta^x_j = \mathcal{O}(a(j))$. Thus, for $0 \leq i \leq m^f(n)-n$,
	\begin{align*}
		\rho^n_{n+i} &= \prod_{j=n}^{n+i-1} \left( 1+  \nabla_x \oh(x^n(t^f(j)), y_n) \right) \rho^n_n \\&+ \sum_{j=n}^{n+i-1} a(j) \prod_{k=j+1}^{n+i-1} \left( 1+  \nabla_x \oh(x^n(t^f(k)), y_n) \right) \left( \nabla_y \oh(x^n(t^f(j)), y_n) \left(y_j - y_n\right) \right. \\&\hspace{18em} \left.+ \kappa^x_j + \kappa^y_j + M^f_{j+1} +\Delta^f_j(x_j)+\delta^x_j \right). 
	\end{align*}
	Defining $\gamma^n_j := \rho^n_j / \sqrt{a(n)}, j \geq n$ and using the fact that $\rho^n_n=0$, we have that
	\begin{align*}
		\gamma^n_{n+i} &= \sum_{j=n}^{n+i-1} \sqrt{a(j)}  \left( \nabla_y \oh(x^n(t^f(j)), y_n) \left(y_j - y_n\right) \right. \\&\hspace{10em} \left. + \kappa^x_j + \kappa^y_j + M^f_{j+1}+\Delta^f_j(x_j) +\delta^x_j \right) \sqrt{\frac{a(j)}{a(j+1)}} \Gamma^{x,n+i-1}_{n,j+1} ,
	\end{align*}
	where $\Gamma^{x,n+i-1}_{n,j+1} =  \prod_{k=j+1}^{n+i-1} \sqrt{\frac{a(k)}{a(k+1)}} \left( 1+  \nabla_x \oh(x^n(t^f(k)), y_n) \right)$. Note that $\Gamma^{x,n+i-1}_{n,j+1} $ is uniformly bounded for all $n$ and $0 \leq i \leq m^f(n)$ by \ref{assumption: alpha_beta} and
	$$
	\begin{aligned}
		\left\|\prod_{r=k+1}^{\ell_1} \left (1 + a(r)\nabla_x\oh(x_{r}, y_{r}) \right )\right\| &\leq e^{\olduck{const: D_xh}\sum_{r=k+1}^{\ell_1} a(r)} \\
		&\leq e^{\olduck{const: D_xh}(T+1)}, 
	\end{aligned}
	$$
	for a suitable bound $\olduck{const: D_xh}$ on $\| \nabla_x \oh(\cdot, \cdot) \| $. Also,
	\begin{align*}
		y_j - y_n = \sum_{k=n}^{j-1} b(k)  \left(g(x_k, y_k,Y_k) + M^s(k+1)\right).	
	\end{align*} 
Thus,
\begin{align*}
&E \left [	\left \| \frac{ y_j - y_{n}}{\sqrt{a(j)}} \right \|^4 \right ] \\
&= E \left [	\left \| \sqrt{\frac{a(n)}{a(j)}}\frac{ y_j - y_{n}}{\sqrt{a(n)}} \right \|^4 \right ] \\
&\leq \frac{1}{a(n)^2} E \left [ \left ( \sqrt{\frac{a(n)}{a(m^f(n))}} \sum^{j-1}_{k=n} \frac{b(k)}{a(k)} a(k) \left( \left \| g(x_k, y_k,Y_k) \right \| + \left \| M^s(k+1) \right \| \right ) \right )^4 \right ]\\
&\overset{(a)}{\leq}  \newuck \frac{\left( \sum^{j-1}_{k=n}  a(k) \right)^4}{a(n)^2} E \left [  \sum^{j-1}_{k=n} \frac{a(k)}{\left( \sum^{j-1}_{k=n}  a(k) \right)} \left (\frac{b(k)}{a(k)} \right )^4 \left( \left \| g(x_k, y_k,Y_k) \right \| + \left \| M^s(k+1) \right \| \right )^4  \right ]\\
	&\overset{(b)}{\leq}	 \newuck \left( \sum^{j-1}_{k=n}  a(k) \right)^3 \underset{n \leq k \leq j}{\sup} \left (\frac{b(k)}{a(k)^{3/2}} \right )^4 \\
	&\leq	 \newuck  \underset{n \leq k \leq j}{\sup} \left (\frac{b(k)}{a(k)^{3/2}} \right )^4 ,
\end{align*}
where $(a)$ follows from \ref{assumption: LemmaA.28} and Jensen's inequality and $(b)$ follows from \ref{assumption: an_decreasing} and the boundedness of $g(\cdot,\cdot,\cdot)$ and the fourth moment of $M^s_{(\cdot)}$. It follows that
\begin{align*}
		&E \left [	\left \| \sum_{j=q}^{k-1} \sqrt{a(j)}  \nabla_y \oh(x^n(t^f(j)), y_n) (y_j-y_n) \sqrt{\frac{a(j)}{a(j+1)}} \Gamma^{x,k-1}_{n,j+1} \right \|^4  \right ] \\
		&= E \left [\left \| \sum_{j=q}^{k-1} a(j)  \nabla_y \oh(x^n(t^f(j)), y_n) \frac{(y_j-y_n)}{\sqrt{a(j)}} \sqrt{\frac{a(j)}{a(j+1)}} \Gamma^{x,k-1}_{n,j+1} \right \|^4  \right ] \\
		&\leq \left(\sum_{j=q}^{k-1} a(j)\right)^3  \sum_{j=q}^{k-1} a(j) E\left [ \left \| \nabla_y \oh(x^n(t^f(j)), y_n) \frac{(y_j-y_n)}{\sqrt{a(j)}} \sqrt{\frac{a(j)}{a(j+1)}} \Gamma^{x,k-1}_{n,j+1} \right \|^4 \right ]  \\
		&\leq \newuck \left(\sum_{j=q}^{k-1} a(j)\right)^3 \sum_{j=q}^{k-1} a(j) \underset{n \leq k \leq j}{\sup} \left (\frac{b(k)}{a(k)^{3/2}} \right )^4 \\
		&\leq \newuck  \left(\sum_{j=q}^{k-1} a(j)\right)^4.
\end{align*}
	As $\kappa^y_j = o(\| y_j-y_n \|)$, we also have that
	\begin{align*}
		E \left[ \left \| \sum_{j=q}^{k-1} \sqrt{a(j)}  \kappa^y_j \sqrt{\frac{a(j)}{a(j+1)}} \Gamma^{x,k-1}_{n,j+1} \right \|^4 \right] 	&\leq \newuck  \left(\sum_{j=q}^{k-1} a(j)\right)^4.
	\end{align*}
	The terms corresponding to $\kappa^x_j, M^f_{j+1},\delta^x_j$ and $\Delta^f_j(y_j)$ can be handled in the same way as in the proof of Lemma \ref{lem: tightness_first_ts}:
	\begin{align*}
		\left \| \sum_{j=q}^{r} \sqrt{a(j)}  \kappa^x_j \sqrt{\frac{a(j)}{a(j+1)}} \Gamma^{x,k-1}_{n,j+1} \right \| &\leq \newuck\sum_{k=q}^{r} a(k) \left \| \gamma^n_k \right \|, \\
		\left \| \sum_{j=q}^{r} \sqrt{a(j)}  \delta^x_j \sqrt{\frac{a(j)}{a(j+1)}} \Gamma^{x,k-1}_{n,j+1} \right \| &\leq \newuck\sum_{k=q}^{r} a(k)^{3/2}\\
		E \left[ \underset{q \leq k \leq r}{\sup } \left \| \sum_{j=q}^{k-1} \sqrt{a(j)}  M^f_{j+1}\sqrt{\frac{a(j)}{a(j+1)}} \Gamma^{x,k-1}_{n,j+1} \right \|^4 \right] &\leq \newuck \left(\sum_{k=q}^{r} a(k) \right)^2. \\
		E \left[  \left \| \sum_{j=q}^{r} \sqrt{a(j)}  \Delta^f_{j}(y_j)\sqrt{\frac{a(j)}{a(j+1)}} \Gamma^{x,k-1}_{n,j+1} \right \|^4 \right] &\leq \newuck \left(\sum_{k=q}^{r} a(k) \right)^2.
	\end{align*}
	Combining the above equations, we have that
	\begin{align}
		E \left[\left \| \gamma^n(t^f(k)) \right \|^4 \right] \leq \newuck \left(\sum_{j=n}^{k} a(k) \right)^2 \left( 1 + \sum_{j=n}^{k} a(j)E \left[ \| \gamma^n(t^f(j)) \|^4 \right] \right), \label{eqn: gamma_discrete_gronwall}
	\end{align}
	using Jensen's inequality. Note that $\left(\sum_{j=n}^{k} a(k) \right)^2 \leq T^2$ for $n \leq k \leq m^f(n)$. It follows from the discrete Gronwall inequality that
$$
	\begin{aligned}
		\underset{n \leq k \leq m^f(n)}{\sup} E \left[\left \| \gamma^n(t^f(k))\right \|^4 \right] < \newuck\label{const: discrete_gronwall_gamma}, \label{eqn: first_ts_CLT_mom4_ub}
	\end{aligned}
$$
where \olduck{const: discrete_gronwall_gamma} does not depend on $n$ but does depend on $T$. 
\end{proof}
\section{Proof of Lemma \ref{lem: slow_tightness}} \label{sec: tightness_slow}
\begin{proof}[Proof of Lemma \ref{lem: slow_tightness}.]
By \ref{assumption: alpha_beta}, $\sqrt{b(j)/b(j+1)}$ is uniformly bounded in $j$. Also, as $f$ is uniformly Lipschitz, $\nabla f$ is uniformly bounded. Thus, for $n \leq k \leq m^s(n)$, 
$$
\begin{aligned}
	\left\|\prod_{r=k+1}^{\ell}(1 + b(r)\nabla \of( y^n(t^s(r) ) ) ) \right\| &\leq e^{  \sum_{r=k+1}^{\ell} a(r)} \\
	&\leq e^{\newuck\label{const: Df}(T+1)},
\end{aligned}
$$
for a suitable bound $\olduck{const: Df} > 0$ on $\|\nabla \of(\cdot)\|$. We define
\begin{align*}
	\Gamma_{n,j}^{s,k} 	:= \prod_{r=j}^{k} \left(I+b(r) \nabla \of( y^n(t^s(r) ) )\right) \sqrt{\frac{a(r)}{a(r+1)}} .
\end{align*}
It follows that the above quantity is bounded. For a sufficiently large $j$, there exists $\newuck\label{const: kappa_2},\newuck\label{const: delta_factor}>0$ such that,
\begin{equation}
	\begin{aligned}
		\| e^{s,T}_j \| &\leq \olduck{const: kappa_2} \|\alpha_j^n \| = \sqrt{a(j)} \olduck{const: kappa_2} \|w_j^n\| \\
		\| \delta_j \| &\leq \olduck{const: delta_factor} b(j).
	\end{aligned}
\end{equation}
Thus for large $j$,
$$
\begin{aligned}
	\left\|\sum_{j=k+1}^{\ell} \frac{b(j)}{\sqrt{a(j)}} \Gamma_{n,j+1}^{s,\ell} e^{s,T}_j \sqrt{\frac{a(j)}{a(j+1)}} \right\| 	&\leq \newuck\label{const: esT_ub} \sum_{j=k+1}^{\ell} \frac{b(j)}{\sqrt{a(j)}} \left( \sqrt{a(j)} \|  w_j^n\|  \right) \\
	&= \olduck{const: esT_ub}\sum_{j=k+1}^{\ell} b(j) \|  w_j^n\|.
\end{aligned}
$$
We also have that for some $\newuck\label{const: delta_jn_final}>0$, using the fact that $\sqrt{b(j)} \leq \sqrt{a(j)}$, 
\begin{align*}
	\left\|\sum_{j=k+1}^{\ell} \frac{b(j)}{\sqrt{a(j)}} \Gamma_{n,j+1}^{s,\ell} \left ( \delta_j \right ) \sqrt{\frac{a(j)}{a(j+1)}} \right\| 
	&\leq \olduck{const: delta_jn_final} \sum_{j=k+1}^{\ell} \frac{b(j)}{\sqrt{a(j)}} b(j) \\
	&\leq \olduck{const: delta_jn_final}  \sum_{j=k+1}^{\ell} b(j)^{\frac{3}{2}} \\
	&\leq \olduck{const: delta_jn_final} \underset{k+1 \leq m < k}{\sup} \sqrt{b(m)} \sum_{j=k+1}^{\ell} b(j).
\end{align*}

Using Lemma \ref{lem: lemma7.1}, we have that for a large enough $\newuck\label{const: martingale_2}$,
\begin{align*}
	& E \left [ \underset{k \leq m \leq \ell}{\sup} \left\|\sum_{j=k+1}^{m}\frac{b(j)}{\sqrt{a(j)}}\Gamma_{n,j+1}^{s,m} \left ( M_{j+1}^s \right ) \sqrt{\frac{a(j)}{a(j+1)}} \right\|^4 \right ] \\
	&\leq \olduck{const: martingale_2} \left( \left( \sum_{j=k+1}^{\ell} \frac{b(j)^2}{a(j)} \right)^2 + \sum_{j=k+1}^{\ell} \frac{b(j)^4}{a(j)^2} \right ) \\
	&\leq \olduck{const: martingale_2} \left( \left( \sum_{j=k+1}^{\ell} b(j) \right)^2 + \sum_{j=k+1}^{\ell} b(j)^2 \right ) \\
	&\leq \olduck{const: martingale_2} \left( \sum_{j=k+1}^{\ell} b(j) \right)^2 .
\end{align*}
Also, using the fact that $\og$ is Lipschitz and Jensen's inequality,
$$
\begin{aligned}
	&\underset{n \leq k \leq m^s(n)}{\sup} E \left [\left\|\sum_{j=k}^{\ell} \frac{b(j)}{\sqrt{a(j)}} \Gamma_{n,j+1}^{s,k} p(x_j, y_j)  \sqrt{\frac{a(j)}{a(j+1)}} \right \|^4 \right ]\\
	&\leq \newuck\label{const: p}  \underset{n \leq k \leq m^s(n)}{\sup} E \left [ \left ( \sum_{j=k}^{\ell} b(j) \frac{\|x_j - \lambda(y_j)\|}{\sqrt{a(j)}} \right )^4\right ] \\
	&= \olduck{const: p}   \underset{n \leq k \leq m^s(n)}{\sup} \left( \sum_{j=k}^{\ell} b(j) \right)^4 E \left [ \left ( \sum_{j=k}^{\ell} \frac{b(j)}{\sum_{j=k}^{\ell} b(j_1)} \frac{\|x_j -\lambda( y_j) \|}{\sqrt{a(j)}} \right )^4\right ] \\
	&\leq \olduck{const: p}  \underset{n \leq k \leq m^s(n)}{\sup} \left( \sum_{j=k}^{\ell} b(j) \right)^3  \sum_{j=k}^{\ell} b(j) E \left [ \left (\frac{\|x_j - \lambda(y_j) \|}{\sqrt{a(j)}} \right )^4\right ] \\
	&\leq \olduck{const: p} \left( \sum_{j=k}^{\ell} b(j) \right)^4,
\end{aligned}
$$
where in the last step we used the result that $E \left [ \|  u_j\|^4 \right ]<\infty$ from Section \ref{sec: tightness_fast}. 

The final term left to handle in \eqref{eqn: alphan_1} is the one involving $\Delta_j^s(x_j)$. Recall \eqref{eqn: delta_decomp_slow}:
\begin{align*}
	\Delta_j^s(x) &= \zeta_{j}^s(x) +\tau^s_{j}(x)-\tau^s_{j+1}(x)+e^{s,\Delta}_{j}(x).
\end{align*}
The rest of this proof will be devoted to handling each of these terms. Using a summation by parts argument,
\begin{align*}
	&\sum_{j=k+1}^{\ell} \frac{b(j)}{\sqrt{a(j)}} \Gamma_{n,j+1}^{s,\ell} \left(\tau^s_{j+1}(x_j) -  \tau_{j}^s(x_j)   \right) \sqrt{\frac{a(j)}{a(j+1)}} \\
	&= \frac{b(\ell)}{\sqrt{a(\ell)}} \tau^s_{\ell+1}(x_{\ell}) - \frac{b(k+1)}{\sqrt{a(k+1)}}\Gamma_{n,k+1}^{s,\ell} \tau^s_{k+1}(x_{k+1}) \\
	&+\sum_{j=k+2}^{\ell} \Gamma_{n,j+1}^{s,\ell} \frac{b(j)}{\sqrt{a(j)}} b(j-1) \left( \nabla \of(y^n(t^s(j))) \right. \\&\hspace{8em} \left.+ \left (\frac{1}{b(j)} - \frac{1}{b(j-1)} \right )I \right) \tau_j^s(x_j) \sqrt{\frac{a(j)}{a(j+1)}}, \numberthis \label{eqn: sum_of_parts_slow}
\end{align*}
where $I :=$ the identity matrix. Thus,
\begin{align*}
	&\left \| \sum_{j=k+1}^{\ell} \frac{b(j)}{\sqrt{a(j)}} \Gamma_{n,j+1}^{s,\ell} \left( \tau_{j}^s(x_j)  - \tau^s_{j+1}(x_j) \right) \sqrt{\frac{a(j)}{a(j+1)}} \right \| \\
	&\leq \newuck \left( \left \| \frac{b(\ell)}{\sqrt{a(\ell)}} \tau^s_{\ell+1}(x_{\ell}) \right \|+ \left \| \frac{b(k+1)}{\sqrt{a(k+1)}}\Gamma_{n,k+1}^{s,\ell} \tau^s_{k+1}(x_{k+1}) \right \| \right.\\
	&\left.+\left \| \sum_{j=k+2}^{\ell} \Gamma_{n,j+1}^{s,\ell} \tau_j^s(x_j) \frac{b(j)}{\sqrt{a(j)}} b(j-1) \left( \nabla \of(y^n(t^s(j))) + \left (\frac{1}{b(j)} - \frac{1}{b(j-1)} \right )I \right) \right \| \right).  \numberthis \label{eqn: tau_decomp_slow}
\end{align*} 
Note that as $V^s(\cdot, \cdot, \cdot)$ is bounded, $\tau^s_j(x_j)$ is bounded. Using $(a+b+c)^4 \leq 27(a^4+b^4+c^4)$, \ref{assumption: alpha_beta} and Jensen's inequality, we have
\begin{align*}
	&E \left [ \left \| \sum_{j=k+1}^{\ell} \frac{b(j)}{\sqrt{a(j)}} \Gamma_{n,j+1}^{s,\ell} \left( \tau_{j}^s(x_j)  - \tau^s_{j+1}(x_j) \right) \sqrt{\frac{a(j)}{a(j+1)}} \right \|^4 \right ] \\
	&\leq \newuck \left( b(\ell)^2 + b(k+1)^2 + E\left [\left ( \sum_{j=k+2}^{\ell} \frac{b(j)}{\sqrt{a(j)}}b(j-1) \left \| \tau_j^s(x_j) \right \| \right )^4 \right]  \right) \\
	&\leq \newuck \left( 2\left(\sum_{j=k+1}^{\ell} b(j)\right)^2 + \left(\sum_{j=k+2}^{\ell} b(j-1)\right)^3 \sum_{j=k+2}^{\ell}  \frac{b(j)^4}{a(j)^2}b(j-1) E \left[ \left \| \tau_j^s(x_j) \right \|^4\right] \right) \\
	&\leq \newuck \left(\sum_{j=k+1}^{\ell} b(j)\right)^2.
\end{align*}
Also, as $V^s(x, \cdot, Y_{n+1})$ is Lipschitz,
$$
\begin{aligned}
	\left \| e^{s,\Delta}_{j}(x) \right \| &= \left \| V^s(x, y_{j+1}, Y_{j+1}) - V^s(x, y_j, Y_{j+1}) \right \| \\
	&\leq \newuck \left \| y_{j+1} - y_j \right \| \\
	&\leq\newuck b(j),
\end{aligned}
$$
as $\og(\cdot, \cdot, \cdot)$ is bounded. Thus,
\begin{align*}
	&E \left [ \left \| \sum_{j=k+1}^{\ell} \frac{b(j)}{\sqrt{a(j)}} \Gamma_{n,j+1}^{s,\ell} e^{s,\Delta}_{j}(x_j) \sqrt{\frac{a(j)}{a(j+1)}} \right \|^4 \right ] \\
	&\leq \newuck \left(\sum_{j=k+1}^{\ell} b(j)\right)^3 \left (\sum_{j=k+1}^{\ell} \frac{b(j)}{a(j)^2} b(j)^4 \right )\\
	&\leq \newuck \left(\sum_{j=k+1}^{\ell} b(j)\right)^2.
\end{align*}
Recall that $\zeta^s_j(x_j)$ is a bounded martingale difference sequence. Thus, using Lemma \ref{lem: lemma7.1} for 
\begin{align*}
	\Psi_{k,m} := \sum_{j=k+1}^{m} \frac{b(j)}{\sqrt{a(j)}} \Gamma_{n,j+1}^{s,m} \zeta^{s}_{j}(x_j) \sqrt{\frac{a(j)}{a(j+1)}},
\end{align*}
we have that
\begin{align*}
	E \left [ \underset{k+1 \leq m \leq \ell}{\sup} \| \Psi_m \|^4 \right ] &\leq \left(\sum_{j=k+1}^{\ell} b(j)\right)^2 + \sum_{j=k+1}^{\ell} b(j)^2 \\
	&\leq \newuck \left(\sum_{j=k+1}^{\ell} b(j)\right)^2.
\end{align*}

Combining the above bounds, we have that for $n \leq k \leq m^s(n)$,
\begin{align*}
	E \left[ \left \| w^n_k \right \|^4 \right] &\leq \newuck \left ( \left(\sum_{j=n}^{k} b(j)\right)^2 + E \left[ \left(\sum_{j=n}^{k} b(j) \|  w_j^n\| \right)^4 \right] \right ) \\
	&\leq \newuck \left(\sum_{j=n}^{k} b(j)\right)^2 \left( 1+ \sum_{j=n}^{k}  b(j) E \left[\| w^n_j \|^4 \right]\right), \numberthis \label{eqn: slow_final_ub_tight}
\end{align*}
where the last step follows from Jensen's inequality. Using the discrete Gronwall inequality, it follows that
\begin{align}
	\underset{n \leq k \leq m^s(n)}{\sup}	E \left[ \left \| w^n_{k} \right \|^4 \right] \leq \newuck < \infty, \label{eqn: slow_fourth_moment_bdd}
\end{align}
where the upper bound can be a function of $T$ but is inddependent of $n$. Using this result in \eqref{eqn: slow_final_ub_tight}, we have that
\begin{align*}
	E \left[ \left \| w_2^n(t^s(\ell)) - w_2^n(t^s(k)) \right \|^4 \right] &\leq \left(\sum_{j=k+1}^{\ell} b(j)\right)^2 \\
	&=O \left( \left | t^s(\ell) - t^s(k) \right |^2 \right).
\end{align*}
\end{proof}
\section{Proof of Lemma \ref{thm: main_thm_alpha}} \label{sec: main_thm_alpha}
\begin{proof}[Proof of Theorem \ref{thm: main_thm_alpha}.]
The proof of this theorem is similar to that of Theorem \ref{thm: main_thm}. Specifically, the proof of tightness of the laws of $\tilde{u}^n(\cdot)$ and characterizing its limit in law is exactly the same. For the second timescale, we will follow the earlier template: first prove tightness for the laws of $\tilde{w}^n(\cdot)$ and then characterize its limit in law. 

\noindent\textbf{Proof of tightness.} Let $y^n(t), t \geq t^f(n)$ be the unique solution of the o.d.e.\ \eqref{eqn: second_timescale_ode} started at $t^f(n)$, i.e.,
\begin{align*}
	\dot{y}^n(t) &= \epsilon_n \of(y^n(t)), t \geq t^f(n), \\
	y^n(t^f(n)) &= y_n.
\end{align*}
Using the Taylor expansion of $y^n(\cdot)$, we have the following iterative equation for $j \geq n$:
\begin{align*}
	y^n(t^f(j+1)) &= y^n(t^f(j)) + a(j) \left(  \epsilon_n \of(y^n(t^f(j))) - \delta^y_j \right),
\end{align*}
where $\delta^y_j = O(a(j))$. Subtracting the above equation from \eqref{eqn: equation2_alpha}, we have
\begin{align*}
	\alpha^n_{j+1} := y_n - y^n(t^f(j+1)) = \alpha^n_j + a(j)\epsilon_j \left( \epsilon_j \nabla \of(y^n(t^f(j))) \alpha^n_j + e^{s,T}_j+ \delta^y_j + p(x_j, y_j) +  \Delta^s_j(x_j) + M^s_{j+1}  \right),
\end{align*}
where $e^{s,T}_j = o(\|\alpha^n_j\|)$ from the Taylor expansion. Then
\begin{align*}
	w^n_{j+1} &= w^n_j \left ( I + a(j) \epsilon_j \nabla \of(y^n(t^f(j)))  \right )\sqrt{\frac{a(j)}{a(j+1)}} \\&+ \sqrt{a(j)} \epsilon_j\left(e^{s,T}_j+ \delta^y_j +  p(x_j, y_j) +  \Delta^s_j(x_j) + M^s_{j+1}  \right)\sqrt{\frac{a(j)}{a(j+1)}}.
\end{align*}
Iterating the above equation for $0 < i \leq m^f(n)-n$, we have that
\begin{align*}
	w^n_{n+i} &= \prod_{j=n}^{n+i-1}   \left ( I + a(j) \epsilon_j \nabla \of(y^n(t^f(J)))  \right )\sqrt{\frac{a(j)}{a(j+1)}} w^n_n \\&+ \sum_{j=n}^{n+i-1} \sqrt{a(j)} \Gamma_{n,j+1}^{s,n+i-1} \epsilon_j \left(e^{s,T}_j+ \delta^y_j +  p(x_j, y_j) + \Delta^s_j(x_j) + M^s_{j+1}  \right)\sqrt{\frac{a(j)}{a(j+1)}},
\end{align*}
where
\begin{align*}
	\Gamma_{n,j+1}^{s,n+i-1} = \prod_{k=j+1}^{n+i-1}   \left ( I + a(k) \epsilon_k \nabla \of(y^n(t^f(k)))  \right )\sqrt{\frac{a(k)}{a(k+1)}}, n \leq j \leq m^f(n). 
\end{align*}
It follows from arguments in Section \ref{sec: tightness_fast} that the above quantity is uniformly bounded. Note that $w^n_n=0$. Next, we prove a result analogous to Lemma \ref{lem: slow_tightness}. 
\begin{lemma}\label{lem: slow_tightness_alpha}
	For $m^s(n) \geq \ell > k \geq n, n \geq 0$,
	\begin{equation*}
		E\left[\|w^n(t^f(\ell)) - w^n(t^f(k))\|^4\right]  =  O\left ( \left (\sum_{j=k}^{\ell}a(j)\right )^2 \right ) 
		=  O(|t^f(\ell) - t^f(k)|^2).
	\end{equation*}
\end{lemma}
\begin{proof}[Proof of Lemma \ref{lem: slow_tightness_alpha}.]
The proof of this lemma is similar to that of Lemma \ref{lem: slow_tightness}. For $n \leq k < \ell \leq m^f(n)$, as $\sum_{j=k}^{\ell} b(k) \leq \sum_{j=k}^{\ell} a(k) \leq T$ and $b(k) \leq a(k)$, the same arguments as in Lemma \ref{lem: slow_tightness} apply. 
Tightness follows from Lemma \ref{lem: Lemma7.4}.
\end{proof}

\noindent\textbf{Characterizing the limit.}
Using the Taylor expansion for $\og$, 
\begin{align*}
	p(x_j, y_j) &= \og(x_j,y_j) - \og(\lambda(y_j),y_j)  \\
	&= \nabla_x \og(\lambda(y_j), y_j) \left(x_j - \lambda(y_j) \right) + e^{s,T,p}_j,
\end{align*}
where $e^{s,T,p}_j = o(\| x_j - \lambda(y_j)\| )$. Recall 
\begin{align*}
	w^n_{j+1} &= w^n_j \sqrt{\frac{a(j)}{a(j+1)}} + a(j) \epsilon_j \nabla \of(y^n(t^f(j)))  \sqrt{\frac{a(j)}{a(j+1)}} \\&+ \sqrt{a(j)} \epsilon_j\left(e^{s,T}_j+ \delta^y_j +  p(x_j, y_j) +  \Delta^s_j(x_j) + M^s_{j+1}  \right)\sqrt{\frac{a(j)}{a(j+1)}} \\
	&= w^n_j \sqrt{\frac{a(j)}{a(j+1)}} + a(j) \epsilon_j \nabla \of(y^n(t^f(j)))  \sqrt{\frac{a(j)}{a(j+1)}} \\&+ \sqrt{a(j)} \epsilon_j\left( \sqrt{a(j)} \nabla_x \og(\lambda(y_j), y_j) u_j+  \Delta^s_j(x_j) + M^s_{j+1}  \right)\sqrt{\frac{a(j)}{a(j+1)}} + o(a(j)).
\end{align*}
Iterating the above equation for $n \leq j<m^s(n)$ gives:
\begin{multline}
	w_{j+1}^n = w_n^n + \sum_{k=n}^j \left( \sqrt{\frac{a(k)}{a(k+1)}} -1 \right) w_k^n + \sum_{k=n}^{j}  a(k) \epsilon_k \nabla \of(y^n(t^s(k))) \sqrt{\frac{a(k)}{a(k+1)}} w_k^n \\
	+ \sum_{k=n}^{j}\frac{a(k)\epsilon_k}{\sqrt{a(k+1)}} \Delta^s_k(x_k)+ \sum_{k=n}^{j} \sqrt{a(k)} \sqrt{\frac{a(k)}{a(k+1)}} \epsilon_k M_{k+1}^s \\+ \sum_{k=n}^j a(k)\epsilon_k \sqrt{\frac{a(k)}{a(k+1)}} {\nabla_x \og(\lambda(y_k), y_k) u_k} + o(1). 
\end{multline}
Using the decomposition for $\Delta^s_k(x_k)$ from \eqref{eqn: delta_decomp_slow}, 
\begin{align*}
	&\sum_{k=n}^j \frac{a(k)\epsilon_k}{\sqrt{a(k+1)}}\Delta^s_k(x_k) = \sum_{k=n}^j \sqrt{a(k)}\frac{\sqrt{a(k)}}{\sqrt{a(k+1)}} \epsilon_k \zeta^s_k(x_k) \\&\hspace{8em}+\sum_{k=n}^j \frac{a(k)\epsilon_k}{\sqrt{a(k+1)}} \left(\tau^s_k - \tau^s_{k+1}\right) + \sum_{k=n}^j a(k) \sqrt{\frac{a(k)}{a(k+1)}}\epsilon_k  \frac{e^{s,\Delta}_k(x_k)}{\sqrt{a(k)}}.  
\end{align*}
Define
\begin{align}
	c(t) &= \sqrt{\frac{a(i)}{a(i+1)}}, t \in [ t^s(i), t^s(i+1) ), i \geq 0, \\
	d_m^n &= \sum_{i=n}^m \left( \sqrt{\frac{a(i)}{a(i+1)}} -1 \right) w_i^n = \sum_{i=n}^m \left( \frac{\varphi}{2} + o(1) \right) w_i^n . 
\end{align}
As shown in Section \ref{sec: limit_slow_martingale},
\begin{align*}
	E \left[ \left \| \sum_{k=n}^j b(k)\sqrt{\frac{a(k)}{a(k+1)}}\frac{\left(\tau^s_k - \tau^s_{k+1}\right)}{\sqrt{a(k)}} \right \|^2 \right] &\overset{n \uparrow \infty}{\rightarrow} 0, \\
		E \left[ \left \| \sum_{k=n}^j b(k)\sqrt{\frac{a(k)}{a(k+1)}}\frac{e^{s,\Delta}_k(x_k)}{\sqrt{a(k)}} \right \|^2 \right] &\overset{n \uparrow \infty}{\rightarrow} 0, \\
\end{align*}
Also, using Jensen's inequality, 
\begin{align*}
	E \left[\left \| \sum_{k=n}^{j} a(k) \sqrt{\frac{a(k)}{a(k+1)}} \frac{b(k)}{a(k)^{3/2}} M_{k+1}^s  \right \|^2 \right] &\leq \newuck \left(\sum_{k=n}^{j} a(k)\right) \sum_{k=n}^{j} a(k) \left(\frac{b(k)}{a(k)^{3/2}} \right)^2 E \left[\|M^s_{k+1} \|^2 \right] \\
	&\leq \newuck \underset{k \geq n}{\sup} \ \left(\frac{b(k)}{a(k)^{3/2}} \right)^2 \rightarrow 0. 
\end{align*}
Similarly,
\begin{align*}
	E \left[\left \| \sum_{k=n}^{j} a(k) \sqrt{\frac{a(k)}{a(k+1)}} \frac{b(k)}{a(k)^{3/2}} \zeta_{k}^s  \right \|^2 \right] &  \rightarrow 0. 
\end{align*}
Also,
\begin{align*}
	E \left[\left \| \sum_{k=n}^j a(k)\epsilon_k \sqrt{\frac{a(k)}{a(k+1)}} {\nabla_x \og(\lambda(y_k), y_k) u_k} \right \|^2 \right] &\leq \newuck \left(\sum_{k=n}^{j} a(k)\right) \sum_{k=n}^{j} a(k) \epsilon_k^2 E \left[\|u_k \|^2 \right] \\
	&\leq \newuck \underset{k \geq n}{\sup}\  \epsilon_k^2  \rightarrow 0,
\end{align*}
using the result from Section \ref{sec: un_bound} that $\{E \left[\| u_n \|^2 \right]\}$ is a bounded sequence. Similarly,
\begin{align*}
	E \left[\left \| \sum_{k=n}^j a(k)\epsilon_k \sqrt{\frac{a(k)}{a(k+1)}} {\nabla \of(y^n(t^f(k))) u_k} \right \|^2 \right] &\leq \newuck \ \underset{k \geq n}{\sup}\  \epsilon_k^2  \rightarrow 0.
\end{align*}

Thus, the resultant limit in law is the following ordinary differential equation: For $t \in [0,T]$,
\begin{align*}
	w^*(t) = \int_{0}^{t} \frac{\varphi}{2} w^*(t) dt.
\end{align*}
\end{proof}

\bibliography{references.bib} 

%

%
%






%
%
%

\end{document}